\documentclass{amsproc}

\usepackage{fullpage}
\usepackage[utf8]{inputenc}
\usepackage[T1]{fontenc}
\usepackage{graphicx,subfigure}
\usepackage{amsmath,amsfonts,amssymb,mathtools}
\usepackage{stmaryrd}
\usepackage{mathrsfs,dsfont}
\usepackage{amsthm}
\usepackage{mathabx}
\usepackage{tabularx} 
\usepackage{float}
\usepackage{amscd}

\usepackage{hyperref}

\usepackage{enumerate}

\usepackage{color}

\newcommand{\E}{\mathbb{E}}
\newcommand{\R}{\mathbb{R}}
\newcommand{\N}{\mathbb{N}}

\newcommand{\T}{\mathbb{T}}

\newcommand{\dd}{\mathrm{d}}

\newtheorem{theo}{Theorem}[section]

\newtheorem{cor}[theo]{Corollary}
\newtheorem{rem}[theo]{Remark}

\newtheorem{propo}[theo]{Proposition}
\newtheorem{lemma}[theo]{Lemma}

\newtheorem{ass}{Assumption}

\usepackage{algorithm}
\usepackage{algorithmic}

\begin{document}

\title{Uniform error bounds for numerical schemes applied to multiscale SDEs in a Wong--Zakai diffusion approximation regime}

\author{Charles-Edouard Br\'ehier}
\address{Univ Lyon, Université Claude Bernard Lyon 1, CNRS UMR 5208, Institut Camille Jordan, 43 blvd. du 11 novembre 1918, F--69622 Villeurbanne cedex, France}
\email{brehier@math.univ-lyon1.fr}

\begin{abstract}
We study a family of numerical schemes applied to a class of multiscale systems of stochastic differential equations. When the time scale separation parameter vanishes, a well-known homogenization or Wong--Zakai diffusion approximation result states that the slow component of the considered system converges to the solution of a stochastic differential equation driven by a real-valued Wiener process, with Stratonovich interpretation of the noise. We propose and analyse schemes for effective approximation of the slow component. Such schemes satisfy an asymptotic preserving property and generalize the methods proposed in the recent article~\cite{BR}. We fill a gap in the analysis of these schemes and prove strong error estimates, which are uniform with respect to the time scale separation parameter.
\end{abstract}

\maketitle

\section{Introduction}

In this article, we consider multiscale systems of stochastic differential equations of the type
\begin{equation}\label{eq:SDEintro}
\left\lbrace
\begin{aligned}
dX^\epsilon(t)&=\frac{\sigma(X^\epsilon(t))m^\epsilon(t)}{\epsilon}dt\\
dm^\epsilon(t)&=-\frac{m^\epsilon(t)}{\epsilon^2}dt+\frac{1}{\epsilon}d\beta(t),
\end{aligned}
\right.
\end{equation}
where $\epsilon\in(0,\epsilon_0)$ is a time-scale separation parameter. On the one hand, $X^\epsilon(t)$ takes values in the $d$-dimensional torus $\T^d=\mathbb{R}^d/(2\pi\mathbb{Z})^d$ in arbitrary dimension $d\in\N$, and $\sigma:\T^d\to\R^d$ is a mapping which is at least of class $\mathcal{C}^3$. On the other hand, the Wiener process $\beta$ and the Ornstein--Uhlenbeck process $m^\epsilon$ are real-valued. The objective of this article is to study numerical schemes for the approximation of the component $X^\epsilon$, for arbitrary values of the time-scale separation parameter $\epsilon$, in particular when it vanishes. This is not a trivial task since the component $m^\epsilon$ evolves at the fast time scale $t/\epsilon^2$, and a crude discretization would impose stringent conditions on the time-step size $\Delta t$.

It is a well-known result in the analysis of multiscale stochastic systems that $X^\epsilon(t)$ converges, at least in distribution, when $\epsilon\to 0$, to $X^0(t)$, for all $t\ge 0$, where $X^0$ is the solution of the stochastic differential equation
\begin{equation}\label{eq:limitingSDEintro}
dX^0(t)=\sigma(X^0(t))\circ d\beta(t)
\end{equation}
where the noise is interpreted in the sense of Stratonovich. We refer for instance to~\cite[Section~11.7.3]{PavliotisStuart} for a description of this convergence result, and see Proposition~\ref{propo:cvSDE} below for a precise statement, where convergence is understood in a stronger sense than convergence in distribution. The convergence result $X^\epsilon\to X^0$ belongs to the class of Wong--Zakai approximation results (the Stratonovich noise $\circ d\beta(t)$ is approximated by a smoother version $d\zeta^\epsilon(t)$), it is also sometimes called a diffusion approximation result.

In order to define numerical schemes which perform better than crude methods when $\epsilon$ varies and may vanish, it is relevant to resort to the notion of asymptotic preserving schemes as studied in the recent article~\cite{BR}: if $\Delta t=T/N$ denotes the time-step size with given $T\in(0,\infty)$ and $N\in\N$, one has a commutative diagram property
\[
\begin{CD}
X_N^{\epsilon,\Delta t}     @>{N \to \infty}>> X^\epsilon(T) \\
@VV{\epsilon\to 0}V        @VV{\epsilon\to 0}V\\
X_N^{0,\Delta t}     @>{N \to \infty}>> X^0(T),
\end{CD}
\]
where $\bigl(X_n^{\epsilon,\Delta t},m_n^{\epsilon,\Delta t}\bigr)_{0\le n\le N}$ is the scheme for given values of $\epsilon$ and $\Delta t$, and one needs to check that
\begin{itemize}
\item the scheme is consistent for any value of $\epsilon>0$ when $\Delta t\to 0$,
\item there exists a limiting scheme $\bigl(X_n^{0,\Delta t}\bigr)_{0\le n\le N}$ when $\epsilon\to 0$ for any value of $\Delta t>0$,
\item the limiting scheme is consistent with the limiting equation when $\Delta t\to 0$.
\end{itemize}
As explained in~\cite{BR}, the last property may fail to hold for some crude methods, for instance using a standard explicit Euler scheme for the discretization of the component $X^\epsilon$ does not capture the It\^o--Stratonovich correction term.

In this article, we study numerical schemes which can be written as
\begin{equation}\label{eq:schemeintro}
\left\lbrace
\begin{aligned}
X_{n+1}^{\epsilon,\Delta t}&=\Phi(\frac{\Delta tm_{n+1}^{\epsilon,\Delta t}}{\epsilon},X_n^{\epsilon,\Delta t})\\
m_{n+1}^{\epsilon,\Delta t}&=m_n^{\epsilon,\Delta t}-\frac{\Delta t}{\epsilon^2}m_{n+1}^{\epsilon,\Delta t}+\frac{\Delta\beta_n}{\epsilon},
\end{aligned}
\right.
\end{equation}
where $\Delta\beta_n=\beta(t_{n+1})-\beta(t_n)$, $t_n=n\Delta t$. The mapping $\Phi$ is an integrator associated with the ordinary differential equation
\[
\frac{dx(t)}{dt}=\sigma(x(t)),
\]
which is assumed to be at least of order $2$, see Assumption~\ref{ass:integrator} below for a precise statement.

Let us state the main result of this article, see Theorem~\ref{theo:main} for a precise statement: one has strong error estimates
\begin{equation}\label{eq:mainintro}
\underset{\epsilon\in(0,\epsilon_0)}\sup~\bigl(\E[\dd(X_N^{\epsilon,\Delta t},X^\epsilon(T))]\bigr)^{\frac1p}\le C_p(T)\Delta t^{\frac12}
\end{equation}
which are uniform with respect to the time-scale separation parameter, for all $p\in[1,\infty)$ and $T\in(0,\infty)$. This means that $X^\epsilon(T)$ can be approximated by $X_N^{\epsilon,\Delta t}$ with a cost which is independent of $\epsilon$. One also checks (see Proposition~\ref{propo:AP}) that the asymptotic preserving property is satisfied, where the convergence results in the diagram above are all understood in the sense of convergence in $L^p(\Omega)$.

Note that the construction of the scheme~\eqref{eq:schemeintro} is inspired by the expression
\[
X^\epsilon(t)=\varphi(\zeta^\epsilon(t),X^\epsilon(0))
\]
for the solution of~\eqref{eq:SDEintro}, where $\zeta^\epsilon(t)=\epsilon^{-1}\int_0^t m^\epsilon(s)ds$, and $\varphi$ is the flow map associated with the ordinary differential equation above. One also has
\[
X^0(t)=\varphi(\beta(t),X(0)),
\]
and the limiting scheme is given by
\[
X_{n+1}^{0,\Delta t}=\Phi(\Delta\beta_n,X_n^{0,\Delta t}).
\]
All the expressions above require to consider real-valued Ornstein--Uhlenbeck process $m^\epsilon$ and Wiener process $\beta$, and also that the evolution of $m^\epsilon(t)$ is independent of the slow component $X^\epsilon(t)$. Even if this situation is restrictive, the analysis requires some non trivial techniques. The analysis of more general situations may be investigated in future works.

Let us give the most crucial arguments of the proof of~\eqref{eq:mainintro}, omitting technicalities. In particular the objective of this discussion is to illustrate why assuming that $\Phi$ is an integrator of order at least $2$ for the associated differential equation above is fundamental. That assumption can be written
\begin{equation}\label{eq:condintro}
\Phi(t,x)-\varphi(t,x)={\rm O}(|t|^3)
\end{equation}
when $|t|\to 0$, see~\eqref{eq:ass_integrator-order} from Assumption~\ref{ass:integrator} for a more precise statement. For an order $1$ method, one would only have ${\rm O}(|t|^2)$ in the right-hand side of~\eqref{eq:condintro}. Using inequalities
\[
\Delta \beta_n={\rm O}(\Delta t^{\frac12}),\quad \frac{\Delta tm_{n+1}^\epsilon}{\epsilon}={\rm O}(\Delta t^{\frac12})
\]
one may already understand why~\eqref{eq:condintro} may be useful to obtain the strong error estimates~\eqref{eq:mainintro} after summation of local error terms of size ${\rm O}(\Delta t^{\frac32})$ to obtain a global error term of size ${\rm O}(\Delta t^{\frac12})$. An additional ingredient of the proof is Lemma~\ref{lem:integrator}, which gives upper bounds
\[
\big|\Phi(t_1+t_2,x)-\Phi(t_1,\Phi(t_2,x)\big|={\rm O}(t_1^2|t_2|+t_2^2|t_1|)
\]
for the integrator (which does not satisfy a flow property, contrary to $\varphi$). Note also that the inequality
\[
\epsilon\|m^\epsilon(t_N)-m_N^{\epsilon,\Delta t}\|={\rm O}(\Delta t^{\frac12})
\]
is used, see Lemma~\ref{lem:cv_m} for a precise statement. The inequality above is combined with~\eqref{eq:condintro} in the proof of~\eqref{eq:mainintro}, again local error terms of size ${\rm O}(\Delta t^{\frac32})$ give a global error term of size ${\rm O}(\Delta t^{\frac12})$. More technical arguments and auxiliary results are required and are omitted in this sketch of proof, see Section~\ref{sec:proofmain} for details.

Note that the recent preprint~\cite{preprintSK} is also concerned with the proof of uniform (strong) error estimates for a class of multiscale SDE systems in a diffusion approximation regime. However, the structure of the systems, the results and the techniques of proof are substantially different, which justifies to perform the analysis in separate articles.

The analysis of numerical methods for multiscale stochastic differential equations is an active research area. The recent articles~\cite{ACLM} and~\cite{Laurent} propose uniformly accurate methods for SDE systems which are different from~\eqref{eq:SDEintro} considered in this article. The recent article~\cite{BR} has introduced a notion of asymptotic preserving schemes which applies to~\eqref{eq:SDEintro}, and some uniform error estimates were proved for SDE systems in an averaging regime. We also refer to the PhD thesis~\cite{RR-thesis} for supplementary results and numerical experiments. In this article, as already mentioned, we fill a gap in~\cite{BR} and prove some uniform error estimates in the diffusion approximation regime, for the scheme~\eqref{eq:schemeintro} applied to~\eqref{eq:SDEintro}. The articles~\cite{FrankGottwald:18} and~\cite{LiAbdulleE:08} illustrate why effective numerical approximation of solutions of SDEs may be more subtle than for deterministic problems. Many other techniques have been introduced to design effective methods for the numerical approximation of multiscale SDE systems, let us mention spectral methods~\cite{AbdullePavliotisVaes:17}, heterogeneous multiscale methods~\cite{ELiuVandenEijnden:05}, projective integration methods~\cite{GivonKevrekidisKupferman:06}, equation-free methods~\cite{KevrekidisAl:03}, parareal methods~\cite{LegollLelievreMyerscoughSamaey:20}, micro-macro acceleration methods~\cite{VandecasteeleZielinskiSamaey:20} for instance. We refer to the monographs~\cite{Gobet,KloedenPlaten,MilsteinTretyakov,
Pages} for general results on numerical methods applied to stochastic differential equations.

This article is organized as follows. Section~\ref{sec:setting} provides the main notation and assumptions. Convergence results, when the time scale separation parameter $\epsilon$ and the time-step size $\Delta t$ vanish are stated in Section~\ref{sec:results}. The main result of this article (strong error estimates in terms of $\Delta t$, uniformly with respect to $\epsilon$, see Theorem~\ref{theo:main}) is stated in Section~\ref{sec:results-main}. Auxiliary results are given in Section~\ref{sec:aux}. Some of the results are proved in Section~\ref{sec:proofs}. The proof of the main result is provided in Section~\ref{sec:proofmain}. Numerical results are reported in Section~\ref{sec:num}.

\section{Setting}\label{sec:setting}

Let $d\in\N$ be an integer. Denote by $\T^d$ the $d$-dimensional torus $\mathbb{R}^d/\mathbb{Z}^d$. If $x_1,x_2\in\T^d$, the distance between $x_1$ and $x_2$ is denoted by $\dd(x_1,d_2)$.

The time-scale separation parameter is denoted by $\epsilon$. Without loss of generality, it is assumed that $\epsilon\in(0,\epsilon_0)$, where $\epsilon_0$ is an arbitrary positive parameter. The time-step size of the numerical schemes is denoted by $\Delta t$. It is assumed that $\Delta t=T/N$ where $T\in(0,\infty)$ is an arbitrary positive real number, and $N\in\N$ is an integer. For all $n\in\{0,\ldots,N\}$, let $t_n=n\Delta t$. Without loss of generality, it is assumed that $\Delta t\in(0,\Delta t_0)$, where $\Delta t_0=T/N_0$ is an arbitrary positive real number. Equivalently, it is assumed that $N\ge N_0$.

Let $\bigl(\beta(t)\bigr)_{t\ge 0}$ be a real-valued standard Wiener process, defined on a probability space $(\Omega,\mathcal{F},\mathbb{P})$ which satisfies the usual conditions. The expectation operator is denoted by $\E[\cdot]$. In the proof of the error estimates, it is convenient to use the following notation. For any real number $p\in[1,\infty)$ and any $\T^d$-valued random variables $\mathcal{X}_1,\mathcal{X}_2$, set
\[
\dd_p(\mathcal{X}_1,\mathcal{X}_2)=\bigl(\E[\dd(\mathcal{X}_1,\mathcal{X}_2)^p]\bigr)^{\frac1p}.
\]
Similarly, for any real number $p\in[1,\infty)$ and any real-valued random variable $\mathcal{Y}$, set
\[
\|\mathcal{Y}\|_p=\bigl(\E[|\mathcal{Y}|^p]\bigr)^{\frac1p}.
\]
The values of constants $C\in(0,\infty)$ may change from line to line in the proofs below. However note that they are always independent of the parameters $\epsilon$ and $\Delta t$.

\subsection{The multiscale SDE system}\label{sec:setting-SDE}

We consider the following class of multiscale stochastic differential equations systems
\begin{equation}\label{eq:SDE}
\left\lbrace
\begin{aligned}
dX^\epsilon(t)&=\frac{\sigma(X^\epsilon(t))m^\epsilon(t)}{\epsilon}dt\\
dm^\epsilon(t)&=-\frac{m^\epsilon(t)}{\epsilon^2}dt+\frac{1}{\epsilon}d\beta(t),
\end{aligned}
\right.
\end{equation}
where $X^\epsilon(t)\in\T^d$ and $m^\epsilon(t)\in\R$, for all $t\ge 0$. The component $m^\epsilon$ is solution of a one-dimensional stochastic differential equation, and is an Ornstein--Uhlenbeck process. It does not depend on the component $X^\epsilon$. The mapping $\sigma$ satisfies Assumption~\ref{ass:sigma} below.
\begin{ass}\label{ass:sigma}
Let $\sigma:\T^d\to \R^d$ be a mapping of class $\mathcal{C}^3$.
\end{ass}
The initial values for the SDE system~\eqref{eq:SDE} are given by $X^\epsilon(0)=x_0^\epsilon$ and $m^\epsilon(0)=m_0^\epsilon$, such that Assumption~\ref{ass:init} below is satisfied.
\begin{ass}\label{ass:init}
There exists $x_0^0\in\T^d$ such that
\[
x_0^\epsilon\underset{\epsilon\to 0}\to x_0^0.
\]
Moreover, one has the following uniform upper bound:
\[
\underset{\epsilon\in(0,\epsilon_0)}\sup~|m_0^\epsilon|<\infty.
\]
\end{ass}
It is assumed that the initial values $x_0^\epsilon\in\T^d$ and $m_0^\epsilon\in\R$ are deterministic. The case of random initial values, independent of the Wiener process $\bigl(\beta(t)\bigr)_{t\ge 0}$, can be treated by a standard conditioning argument, provided that suitable moment bounds are satisfied. This treatment is omitted in the sequel. Note that the constants appearing in the estimates below may depend on the value of $\underset{\epsilon\in(0,\epsilon_0)}\sup~|m_0^\epsilon|$, however this dependence is not indicated explicitly.

It is straightforward to check that the SDE system~\eqref{eq:SDE} admits a unique global solution. Let us give an expression of this solution which plays a crucial role in this article, for the construction and the analysis of numerical schemes which are effective when $\epsilon$ varies and vanishes.

For all $t\ge 0$, define
\begin{equation}\label{eq:zeta}
\zeta^\epsilon(t)=\frac{1}{\epsilon}\int_{0}^{t}m^\epsilon(s)ds,
\end{equation}
where the Ornstein--Uhlenbeck process $m^\epsilon$ is given by
\begin{equation}\label{eq:mepsilon}
m^\epsilon(t)=e^{-\frac{t}{\epsilon^2}}m_0^\epsilon+\frac{1}{\epsilon}\int_{0}^{t}e^{-\frac{t-s}{\epsilon^2}}d\beta(s).
\end{equation}
Let $(t,x)\in\R\times\T^d\mapsto \varphi(t,x)$ be the flow map associated with ordinary differential equation (ODE)
\begin{equation}\label{eq:ODE}
\frac{dx(t)}{dt}=\sigma(x(t)),
\end{equation}
meaning that for all $t\in\R$ and $x\in\T^d$ one has
\begin{equation}\label{eq:ODEvarphi}
\partial_t\varphi(t,x)=\sigma(\varphi(t,x)).
\end{equation}
Since $\sigma$ is Lipschitz continuous by Assumption~\ref{ass:sigma}, the mapping $\varphi$ is well-defined. The evolution equation for $X^\epsilon$ can be interpreted as
\[
dX^\epsilon(t)=\sigma(X^\epsilon(t))d\zeta^\epsilon(t)
\]
using the definition~\eqref{eq:zeta} of $\zeta^\epsilon$. Owing to the definition~\eqref{eq:ODEvarphi} of $\varphi$ and to the chain rule, the unique global solution of~\eqref{eq:SDE}, which can be formulated as follows: for all $t\ge 0$ one has
\begin{equation}\label{eq:SDE-solution}
X^\epsilon(t)=\varphi(\zeta^\epsilon(t),x_0^\epsilon).
\end{equation}

It is worth noting that the expression~\eqref{eq:SDE-solution} and the arguments above require the definition of the flow map $\varphi$ for all values $t\in\R$ in the real line of the time variable. Finally, recall that the flow map $\varphi$ satisfies the following flow property: for all $t_1,t_2\in\R$, one has
\begin{equation}\label{eq:flow-varphi}
\varphi(t_1+t_2,\cdot)=\varphi(t_2,\varphi(t_1,\cdot)).
\end{equation}

\subsection{The numerical scheme}\label{sec:setting-scheme}

Let us now describe the class of proposed numerical schemes. The objective is to approximate the slow component $X^\epsilon$ solving the SDE~\eqref{eq:SDE}, uniformly with respect to the time-scale parameter $\epsilon\in(0,\epsilon_0)$. The approximation needs to be consistent for any fixed $\epsilon$, and to avoid stringent stability conditions depending on $\epsilon$ on the time-step size $\Delta t$. The definition of the numerical scheme is inspired by the expression~\eqref{eq:SDE-solution} for $X^\epsilon(t)$, and is based on a mapping $\Phi:\R\times\T^d\to\T^d$ called the integrator, associated with the ODE~\eqref{eq:ODE}. Let us first state the main conditions required on the integrator.

\begin{ass}\label{ass:integrator}
Let $\Phi:\R\times\T^d\to\T^d$ be a mapping of class $\mathcal{C}^3$, such that there exists $C\in(0,\infty)$ such that for all $t\in\R$ and all $x\in\T^d$, one has
\begin{equation}\label{eq:ass_integrator-order}
\dd\bigl(\Phi(t,x),\varphi(t,x)\bigr)\le C|t|^3e^{C|t|},
\end{equation}
and for all $t_1,t_2\in \R$ and $x_1,x_2\in\T^d$, one has
\begin{equation}\label{eq:ass_integrator-Lip}
\dd\bigl(\Phi(t_2,x_2),\Phi(t_1,x_1)\bigr)\le C\bigl(|t_2-t_1|+\dd(x_2,x_1)\bigr)e^{C(|t_1|+|t_2|)}.
\end{equation}
Finally, for all $i,j\in\{0,1,2,3\}$ with $i+j\le 3$, one has for all $t\in\R$, the upper bound
\begin{equation}\label{eq:ass_integrator-derivees}
\underset{x\in\R^d}\sup~|\partial_t^i\partial_x^j\Phi(t,x)|\le Ce^{C|t|}.
\end{equation}
\end{ass}
If the integrator $\Phi$ satisfies Assumption~\eqref{ass:integrator}, the numerical scheme
\begin{equation}\label{eq:ODEscheme}
x_{n+1}=\Phi(\Delta t,x_n)
\end{equation}
provides an approximation $x_N$ of the solution $\varphi(t_N,x_0)$ of the ordinary differential equation~\eqref{eq:ODE}, when the time-step size $\Delta t>0$ vanishes. On the one hand, the conditions~\eqref{eq:ass_integrator-Lip} and~\eqref{eq:ass_integrator-derivees} are technical requirements for the analysis below. They are variants of properties satisfied by the flow map $\varphi$. On the other hand, the condition~\eqref{eq:ass_integrator-order} is fundamental: it means that the numerical scheme~\eqref{eq:ODEscheme} is (at least) a second-order scheme for the approximation of the solution of the ODE~\eqref{eq:ODE}: for any initial value $x_0\in\T^d$ and all $T\in(0,\infty)$, there exists $C(T,x_0)\in(0,\infty)$ such that
\[
\dd(x_N,\varphi(t_N,x_0))\le C(T,x_0)\Delta t^2,
\]
where we recall that $T=N\Delta t$. As will explained below, the second-order accuracy of the integrator $\Phi$ related to the ODE~\eqref{eq:ODE} plays a crucial role in this article.

Let us give some examples of integrators which satisfy Assumption~\ref{ass:integrator}: introduce
\begin{itemize}
\item the second-order Taylor scheme: $\Phi(t,x)=x+t\sigma(x)+\frac{t^2}{2}\sigma'(x)\sigma(x)$,
\item the explicit midpoint scheme: $\Phi(t,x)=x+t\sigma(x+\frac{t}{2}\sigma(x))$,
\item Heun's method: $\Phi(t,x)=x+\frac{t}{2}\bigl(\sigma(x)+\sigma(x+t\sigma(x))\bigr)$,
\end{itemize}
for all $t\in\R$ and $x\in\T^d$. For the three examples above, the exponential growth with respect to $t_1,t_2$ in the right-hand side of~\eqref{eq:ass_integrator-Lip} and with respect to $t$ in the right-hand side of~\eqref{eq:ass_integrator-derivees} can be replaced by a polynomial growth. Obviously, $\Phi(t,x)=\varphi(t,x)$ is also a possible choice, if the flow map $\varphi$ is known.

However, note that standard explicit Euler integrator, given by $\Phi^{\rm E}(t,x)=x+t\sigma(x)$, does not satisfy Assumption~\ref{ass:integrator}. Indeed it is in general a first-order integrator (except if $\sigma$ is constant), and does not satisfy~\eqref{eq:ass_integrator-order}: instead one only has
\[
\dd\bigl(\Phi^{\rm E}(t,x),\varphi(t,x)\bigr)\le C|t|^2e^{C|t|}
\]
for all $t\in\R$ and $x\in\T^d$.

One may also use splitting schemes to define integrators. Assume that $\sigma=\sigma_1+\sigma_2$ and let $\varphi_1$ and $\varphi_2$ be the flow maps associated with the ODEs $\frac{dx_1(t)}{dt}=\sigma_1(x(t))$ and $\frac{dx_2(t)}{dt}=\sigma_2(x(t))$ respectively. The Strang splitting integrator $\Phi(t,\cdot)=\varphi_2(\frac{t}{2},\cdot)\circ\varphi_1(t,\cdot)\circ\varphi_2(\frac{t}{2},\cdot)$ also satisfies Assumption~\ref{ass:integrator}. However, in general the Lie--Trotter splitting integrator $\Phi^{\rm LT}(t,\cdot)=\varphi_2(t,\cdot)\circ\varphi_1(t,\cdot)$ is only of order $1$ and does not satisfy Assumption~\ref{ass:integrator}, like the standard explicit Euler integrator.

We are now in position to introduce the scheme studied in this article, for the approximation of the solution of the SDE~\eqref{eq:SDE}:
\begin{equation}\label{eq:scheme}
\left\lbrace
\begin{aligned}
X_{n+1}^{\epsilon,\Delta t}&=\Phi(\frac{\Delta tm_{n+1}^{\epsilon,\Delta t}}{\epsilon},X_n^{\epsilon,\Delta t})\\
m_{n+1}^{\epsilon,\Delta t}&=m_n^{\epsilon,\Delta t}-\frac{\Delta t}{\epsilon^2}m_{n+1}^{\epsilon,\Delta t}+\frac{\Delta\beta_n}{\epsilon}
\end{aligned}
\right.
\end{equation}
where $\Delta\beta_n=\beta(t_{n+1})-\beta(t_n)$, with $t_n=n\Delta t$. The initial values are given by $X_0^{\epsilon,\Delta t}=x_0^\epsilon=X^\epsilon(0)$ and $m_0^{\epsilon,\Delta t}=m_0^\epsilon=m^\epsilon(0)$, thus they do not depend on the time-step size $\Delta t$.

The construction of the scheme~\eqref{eq:SDE} is motivated by the identity
\[
X^\epsilon(t_{n+1})=\varphi\bigl(X^\epsilon(t_n),\zeta^\epsilon(t_{n+1})-\zeta^\epsilon(t_n)\bigr)=\varphi\bigl(X^\epsilon(t_n),\frac{1}{\epsilon}\int_{t_n}^{t_{n+1}}m^\epsilon(s)ds\bigr)
\]
which is satisfied by the solution of the SDE~\eqref{eq:SDE}, for all $n\in\{0,\ldots,N-1\}$. The first equality follows from the expression~\eqref{eq:SDE-solution} and the flow property~\eqref{eq:flow-varphi}, whereas the second equality follows from the definition~\eqref{eq:zeta} of the process $\zeta^\epsilon$. The construction of the scheme~\eqref{eq:scheme} is a straightforward combination of two approximations. On the one hand, the flow map $\varphi$ is replaced by the integrator $\Phi$. On the other hand, the component $m^\epsilon$ is discretized using an implicit Euler scheme, and the integral $\int_{t_n}^{t_{n+1}}m^\epsilon(s)ds$ is approximated by a simple quadrature rule (which is adapted to the choice of the implicit Euler scheme). Even if the discretization of the component $m^\epsilon$ is implicit, in practice it is explicit using the expression
\[
m_{n+1}^\epsilon=\frac{1}{1+\frac{\Delta t}{\epsilon^2}}\bigl(m_n^\epsilon+\frac{\Delta\beta_n}{\epsilon}\bigr).
\]

The scheme~\eqref{eq:scheme} is a generalization of the scheme given in~\cite[Corollary~3.15]{BR}, which corresponds to choosing Heun's integrator for the mapping $\Phi$.

\section{Convergence results}\label{sec:results}

In Section~\ref{sec:results-epsilon}, we state results concerning the behavior of $X^\epsilon(t)$ and $X_n^{\epsilon,\Delta t}$ when the time scale separation parameter $\epsilon$ vanishes: we exhibit a limiting stochastic differential equation and a limiting numerical scheme respectively. Then, in Section~\ref{sec:results-main}, we state the main result of this article, giving strong error estimates, in terms of the time-step size $\Delta t$, uniformly with respect to the time scale separation parameter $\epsilon$. Finally, in Sections~\ref{sec:results-lim} and~\ref{sec:results-AP}, we give consequences of the main result, and show that the numerical scheme~\eqref{eq:scheme} is asymptotic preserving. The proofs of the results in Section~\ref{sec:results-epsilon} are postponed to Section~\ref{sec:proofs}, since they exploit some of the auxiliary results from Section~\ref{sec:aux}. The proof of the main result, Theorem~\ref{theo:main}, stated in Section~\ref{sec:results-main}, is postponed to Section~\ref{sec:proofmain}.

\subsection{Asymptotic behavior when the time scale separation vanishes}\label{sec:results-epsilon}

In this section, we study the behavior of the $X^\epsilon(t)$ and $X_n^\epsilon$ when the time scale separation parameter $\epsilon$ vanishes and describe associated evolution equations for the limits.

Let us first focus on the behavior of $X^\epsilon(t)$. Introduce the stochastic differential equation
\begin{equation}\label{eq:limitingSDE}
dX^0(t)=\sigma(X^0(t))\circ d\beta(t)
\end{equation}
with initial value $X^0(0)=x_0^0\in\T^d$ (given in Assumption~\ref{ass:init}). The noise is interpreted in the Stratonovich sense, and the SDE can be written in It\^o form
\[
dX(t)=\frac12\sigma'(X(t))\sigma(X(t))dt+\sigma(X(t))d\beta(t).
\]
The solution of the SDE~\eqref{eq:limitingSDE} is expressed as follows, using the flow map $\varphi$ introduced in Section~\ref{sec:setting-SDE}: for all $t\ge 0$, one has
\begin{equation}\label{eq:limitingSDE-solution}
X(t)=\varphi(\beta(t),x_0^0).
\end{equation}
The identity~\eqref{eq:limitingSDE-solution} is a straightforward consequence of the chain rule associated with the Stratonovich convention and of the definition of $\varphi$ as the flow map associated with the ODE~\eqref{eq:ODE}. One has the following result.
\begin{propo}\label{propo:cvSDE}
Let Assumptions~\ref{ass:sigma} and~\ref{ass:init} be satisfied. For all $T\in(0,\infty)$ and $p\in[1,\infty)$, there exists $C_p(T)\in(0,\infty)$ such that
\[
\underset{0\le t\le T}\sup~\dd_p(X^\epsilon(t),X(t))\le C_p(T)\bigl(\epsilon+\dd(x_0^\epsilon,x_0^0)\bigr)\underset{\epsilon\to 0}\to 0.
\]
\end{propo}

Let us now focus on the behavior of $X_n^{\epsilon,\Delta t}$. Introduce the scheme defined by
\begin{equation}\label{eq:limitingscheme}
X_{n+1}^{0,\Delta t}=\Phi(\Delta \beta_n,X_n^{0,\Delta t}),
\end{equation}
for all $n\in\{0,\ldots,N-1\}$, with initial value $X_0^{0,\Delta t}=x_0^0=X^0(0)$, where we recall that one has $\Delta\beta_n=\beta(t_{n+1})-\beta(t_n)$. One has the following result.
\begin{propo}\label{propo:cvScheme}
Let Assumptions~\ref{ass:sigma},~\ref{ass:init} and~\ref{ass:integrator} be satisfied. For all $p\in[1,\infty)$, all $\Delta t=T/N\in(0,\Delta t_0)$ and all $n\in\{0,\ldots,N\}$, one has
\begin{equation}\label{eq:cvScheme}
\dd_p(X_n^{\epsilon,\Delta t},X_n^{0,\Delta t})\underset{\epsilon\to 0}\to 0.
\end{equation}
\end{propo}
We refer to Section~\ref{sec:proofs} for the proofs of Propositions~\ref{propo:cvSDE} and~\ref{propo:cvScheme}. Note that the condition~\eqref{eq:ass_integrator-order} from Assumption~\ref{ass:integrator} is not used to prove Proposition~\ref{propo:cvScheme}. When choosing different examples for the integrator $\Phi$, one recovers standard numerical schemes from the limiting scheme~\eqref{eq:limitingscheme} for the approximation of the limiting stochastic differential equation~\eqref{eq:limitingSDE}, see for instance~\cite[Part~V]{KloedenPlaten}.

\subsection{Uniform strong error estimates for the scheme~\eqref{eq:scheme}}\label{sec:results-main}

We are now in position to state the main result of this article.
\begin{theo}\label{theo:main}
Let Assumptions~\ref{ass:sigma},~\ref{ass:init} and~\ref{ass:integrator} be satisfied. For all $T\in(0,\infty)$ and $p\in[1,\infty)$, there exists $C_p(T)\in(0,\infty)$ such that for all $\Delta t\in(0,\Delta t_0)$, one has
\begin{equation}\label{eq:main}
\underset{\epsilon\in(0,\epsilon_0)}\sup~\dd_p\bigl(X_N^{\epsilon,\Delta t},X^\epsilon(T)\bigr)\le C_p(T)\Delta t^{\frac12}.
\end{equation}
\end{theo}

Before proceeding with the description of some consequences and the proof of Theorem~\ref{theo:main}, let us emphasize that the condition~\eqref{eq:ass_integrator-order} from Assumption~\ref{ass:integrator} is required for Theorem~\ref{theo:main} to hold. Indeed, Theorem~\ref{theo:main} does not hold when the integrator $\Phi$ is the standard explicit Euler integrator (if $\sigma$ is not constant). Recall that
\[
\Phi^{\rm E}(t,x)=x+t\sigma(x)
\]
for all $t\in\R$ and $x\in\T^d$. Using $\Phi=\Phi^{\rm E}$ in the definitions of the schemes~\eqref{eq:scheme} and~\eqref{eq:limitingscheme} gives the limiting scheme
\[
X_{n+1}^{0,\Delta t,{\rm E}}=X_{n}^{0,\Delta t,{\rm E}}+\Delta \beta_n\sigma(X_n^{0,\Delta t,{\rm E}})
\]
with initial value $X_0^{0,\Delta t,{\rm E}}=x_0^0$. This is the standard Euler--Maruyama scheme applied to the stochastic differential equation
\[
dX^{0,{\rm E}}(t)=\sigma(X^{0,{\rm E}}(t))d\beta(t)
\]
with initial value $X^{0,{\rm E}}(0)=x_0^0$, where the noise is interpreted in the It\^o sense. It is a well-known result that the standard Euler--Maruyama scheme applied in this setting is a method of strong order $1/2$: for all $p\in[1,\infty)$, there exists $C_p(T)\in(0,\infty)$ such that for all $\Delta t\in(0,\Delta t_0)$ one has
\[
\dd_p\bigl(X_N^{0,\Delta t,{\rm E}},X^{0,{\rm E}}(T)\bigr)\le C_p(T)\Delta t^{\frac12}
\]
In general, since $\sigma$ is not constant, one has $X^{0,{\rm E}}(T)\neq X^0(T)$, therefore
\[
\underset{\epsilon\in(0,\epsilon_0)}\sup~\dd_p\bigl(X_N^{\epsilon,\Delta t,{\rm E}},X^\epsilon(T)\bigr)
\]
does not converge to $0$ when $\Delta t\to 0$. If the condition~\eqref{eq:ass_integrator-order} from Assumption~\ref{ass:integrator} is not satisfied, one can only obtain strong error estimates
\[
\dd_p\bigl(X_N^{\epsilon,\Delta t,{\rm E}},X^\epsilon(T)\bigr)\le C_p(\epsilon,T)\Delta t^{\frac12}
\]
which are not uniform with respect to $\epsilon$: one has $\underset{\epsilon\in(0,\epsilon_0)}\sup~C_p(\epsilon,T)=\infty$.

The arguments above illustrate why the uniform strong error estimates~\eqref{eq:main} from Theorem~\ref{theo:main} are non trivial results. The proof given in Section~\ref{sec:proofmain} illustrates the role of the condition~\eqref{eq:ass_integrator-order} from Assumption~\ref{ass:integrator}. We refer to Section~\ref{sec:num} for numerical experiments which illustrate Theorem~\ref{theo:main} and the discussion above.

\subsection{Strong error estimates for the limiting scheme~\eqref{eq:limitingscheme}}\label{sec:results-lim}

As a corollary of Theorem~\ref{theo:main}, letting $\epsilon\to 0$ in the uniform strong error estimate~\eqref{eq:main}, one checks that the limiting scheme~\eqref{eq:limitingscheme} is an integrator of order at least $1/2$ for the limiting SDE~\eqref{eq:limitingSDE}.
\begin{cor}\label{cor:ms}
Let Assumptions~\ref{ass:sigma} and~\ref{ass:integrator} be satisfied. For all $T\in(0,\infty)$ and $p\in[1,\infty)$, there exists $C_p(T)\in(0,\infty)$ such that for all $\Delta t\in(0,\Delta t_0)$, one has
\begin{equation}\label{eq:cor-ms}
\dd_p\bigl(X_N^{0,\Delta t},X^0(T)\bigr)\le C_p(T)\Delta t^{\frac12}.
\end{equation}
\end{cor}
The proof of Corollary~\ref{cor:ms} is a straightforward consequence of Propositions~\ref{propo:cvSDE} and~\ref{propo:cvScheme} and of Theorem~\ref{theo:main}, the details are omitted.

Note that Corollary~\ref{cor:ms} is not a new result and  can be proved directly by applying the fundamental theorem for mean-square convergence when $p=2$, see for instance~\cite[Theorem~1.1]{MilsteinTretyakov}. The details are omitted. Under additional regularity conditions on the mapping $\sigma$ and on the integrator $\Phi$, it is possible to improve the result and obtain strong order of convergence $1$ in the mean-square sense.
\begin{propo}\label{propo:orderlimitingscheme}
Assume that for all $x\in\T^d$ the mappings $\varphi(\cdot,x)$ and $\Phi(\cdot,x)$ are of class $\mathcal{C}^4$, and that there exists $C\in(0,\infty)$ such that for all $t\in\R$ one has
\begin{align*}
\underset{x\in \T^d}\sup~|\partial_t^3\varphi(t,x)|+\underset{x\in \T^d}\sup~|\partial_t^4\varphi(t,x)|+
\underset{x\in \T^d}\sup~|\partial_t^3\Phi(t,x)|+\underset{x\in \T^d}\sup~|\partial_t^4\Phi(t,x)|&\le e^{C|t|}.
\end{align*}
For all $T\in(0,\infty)$, there exists $C_2(T)\in(0,\infty)$ such that for all $\Delta t\in(0,\Delta t_0)$, one has the mean-square error estimate
\begin{equation}\label{eq:orderlimitingscheme}
\underset{0\le n\le N}\sup~\dd_2\bigl(X_n^{\Delta t},X(n\Delta t)\bigr)\le C_2(T)\Delta t
\end{equation}
for the limiting scheme~\eqref{eq:limitingscheme} applied to the limiting SDE~\eqref{eq:limitingSDE}.
\end{propo}

\begin{proof}
In order to apply the fundamental theorem for mean-square convergence, see~\cite[Theorem~1.1]{MilsteinTretyakov}, it suffices to check that there exists $C\in(0,\infty)$ such that for all $x\in\T^d$ and all $\Delta t\in(0,\Delta t_0)$, one has
\begin{align*}
\dd_2\bigl(\varphi(\sqrt{\Delta t}\gamma,x),\Phi(\sqrt{\Delta t}\gamma,x)\bigr)&\le C\Delta t^{\frac32}\\
\dd\bigl(\E[\varphi(\sqrt{\Delta t}\gamma,x)],\E[\Phi(\sqrt{\Delta t}\gamma,x)]\bigr)&\le C\Delta t^2,
\end{align*}
where $\gamma\sim\mathcal{N}(0,1)$ is a standard real-valued Gaussian random variable.

The first claim is a straightforward consequence of the inequality~\eqref{eq:ass_integrator-order} from Assumption~\ref{ass:integrator}, using the fact that $\E[|\gamma|^3 e^{q|\gamma|}]<\infty$ for any positive real number $q\in[0,\infty)$.

The second claim follows from a Taylor-expansion argument (which requires the additional regularity conditions stated in Proposition~\ref{propo:orderlimitingscheme}), using the identity $\E[\gamma^3]=0$ and the fact that $\E[|\gamma|^4 e^{q|\gamma|}]<\infty$ for any positive real number $q\in[0,\infty)$.
\end{proof}

Proposition~\ref{propo:orderlimitingscheme} suggest that the order $1/2$ obtained in Theorem~\ref{theo:main} may not be optimal, under appropriate regularity conditions. Whether one can replace order $1/2$ by order $1$ in the uniform strong error estimates of Theorem~\ref{theo:main} is a question left open for future work.

\subsection{Asymptotic preserving property}\label{sec:results-AP}

As a consequence of Theorem~\ref{theo:main}, one also obtains that the numerical scheme~\eqref{eq:SDE} is asymptotic preserving.
\begin{propo}\label{propo:AP}
For all $p\in[1,\infty)$ and $T\in(0,\infty)$, one has
\begin{equation}\label{eq:propoAP}
\underset{\epsilon\to 0}\lim~\underset{\Delta t\to 0}\lim~\dd_p(X_N^{\epsilon,\Delta t},X^\epsilon(T))=\underset{\Delta t\to 0}\lim~\underset{\epsilon\to 0}\lim~\dd_p(X_N^{\epsilon,\Delta t},X^\epsilon(T))=0.
\end{equation}
\end{propo}
Note that the result differs from the ones considered in the recent article~\cite{BR}, where the asymptotic preserving property is understood in the sense of convergence in distribution. As already explained in~\cite{BR} and using the arguments given above in Section~\ref{sec:results-main}, Proposition~\ref{propo:AP} does not hold when the mapping $\Phi$ is the standard explicit Euler integrator, again the role of the condition~\ref{eq:ass_integrator-order} from Assumption~\ref{ass:integrator} is required to have Proposition~\ref{propo:AP}. Note that in~\cite{BR} only the Heun integrator was considered.

\begin{proof}
On the one hand, Theorem~\ref{theo:main} implies
\[
\underset{\epsilon\to 0}\lim~\underset{\Delta t\to 0}\lim~\dd_p(X_N^{\epsilon,\Delta t},X^\epsilon(T))=\underset{\epsilon\to 0}\lim~0=0.
\]
On the other hand, applying Proposition~\ref{propo:cvSDE} and Proposition~\ref{propo:cvScheme} in the first step, and Corollary~\ref{cor:ms} in the second step, one obtains
\[
\underset{\Delta t\to 0}\lim~\underset{\epsilon\to 0}\lim~\dd_p(X_N^{\epsilon,\Delta t},X^\epsilon(T))=\underset{\Delta t\to 0}\lim~\dd_p(X_N^{0,\Delta t},X^0(T))=0.
\]
The proof of Proposition~\ref{propo:AP} is thus completed.
\end{proof}

\section{Auxiliary results}\label{sec:aux}

The objective of this section is to state and prove some auxiliary results which are used to prove Theorem~\ref{theo:main} in Section~\ref{sec:proofmain}. Let us emphasize that the most important results are Lemma~\ref{lem:integrator}, which is a consequence of the condition~\eqref{eq:ass_integrator-order} from Assumption~\ref{ass:integrator}, and Lemma~\ref{lem:cv_m}. The other auxiliary results are also needed but play less important roles in the analysis.

\subsection{Properties of the flow map and of the integrator}\label{sec:aux-flowintegrator}

Lemma~\ref{lem:flow} below is a standard result in the analysis of ordinary differential equations. Its proof is given for completeness. It is worth comparing the result of Lemma~\ref{lem:flow} for the flow map $\varphi$ with the conditions imposed in Assumption~\ref{ass:integrator} for the integrator $\Phi$. 
\begin{lemma}\label{lem:flow}
There exists $C\in(0,\infty)$ such that for all $x_1,x_2\in \T^d$ and $t_1,t_2\in\R$, one has
\[
\dd\bigl(\varphi(t_1,x_1),\varphi(t_2,x_2)\bigr)\le C\bigl(|t_2-t_1|+e^{C(|t_1|+|t_2|)}|x_2-x_1|\bigr).
\]
\end{lemma}

\begin{proof}[Proof of Lemma~\ref{lem:flow}]
First, owing to~\eqref{eq:ODEvarphi}, if $x\in\T^d$ and $t_1,t_2\in\R$, one has the equality
\[
\varphi(t_2,x)-\varphi(t_1,x)=\int_{t_1}^{t_2}\sigma(\varphi(t,x))dt
\]
and since $\sigma$ is bounded by Assumption~\ref{ass:sigma}, one obtains the inequality
\[
\dd\bigl(\varphi(t_1,x),\varphi(t_2,x)\bigr)\le C|t_2-t_1|.
\]
Second, if $x_1,x_2\in\T^d$, then for all $t\ge 0$ let $\eta(t;x_1,x_2)=\dd\bigl(\varphi(t,x_1),\varphi(t,x_2)\bigr)$. One has
\[
\eta(t;x_1,x_2)\le\int_0^{t}\underset{x\in\T^d}\sup~\|\sigma'(x)\|\eta(s;x_1,x_2)ds\le C\int_0^{t}\eta(s;x_1,x_2)ds,
\]
since $\sigma$ is of class $\mathcal{C}^1$ with bounded derivative by Assumption~\ref{ass:sigma}. Using the identity $\eta(0;x_1,x_2)=\dd(x_1,x_2)$, and applying Gronwall's lemma, one obtains for all $t\ge 0$ the inequality
\[
\dd\bigl(\varphi(t,x_1),\varphi(t,x_2)\bigr)=\eta(t;x_1,x_2)\le e^{Ct}\dd(x_1,x_2).
\]
Similarly, when $t\le 0$, one obtains
\[
\dd\bigl(\varphi(t,x_1),\varphi(t,x_2)\bigr)\le e^{C|t|}\dd(x_1,x_2).
\]

Combining the upper bounds and using the triangle inequality then concludes the proof of Lemma~\ref{lem:flow}.
\end{proof}

Before stating the next auxiliary result, observe that the following identities hold, as a consequence of the condition~\eqref{eq:ass_integrator-order} from Assumption~\ref{ass:integrator}: for all $x\in\R^d$,
\begin{equation}\label{eq:deriveesPhi}
\begin{aligned}
\partial_t\Phi(t=0,x)&=\partial_t\varphi(t=0,x)=\sigma(x)\\
\partial_t^2\Phi(t=0,x)&=\partial_t^2\varphi(t=0,x)=\sigma'(x)\sigma(x)\\
\partial_t\partial_x\Phi(t=0,x)&=\partial_t\partial_x\varphi(t=0,x)=\sigma'(x).
\end{aligned}
\end{equation}
The first and the third equalities in~\eqref{eq:deriveesPhi} only require the integrator to be of order at least $1$, however the assumption~\eqref{eq:ass_integrator-order} that the order is at least $2$ is required to obtain the second equality, which plays a crucial role in the proof of Lemma~\ref{lem:integrator} below.

Let us define the auxiliary function $\delta\Phi$ as follows: for all $t_1,t_2\in\R$ and $x\in\T^d$, set
\begin{equation}\label{eq:deltaPhi}
\delta\Phi(t_1,t_2,x)=\Phi(t_1+t_2,x)-\Phi(t_1,\Phi(t_2,x)).
\end{equation}
In general, $\delta\Phi(t_1,t_2,x)\neq 0$ since the integrator does not satisfy a flow property similar to~\eqref{eq:flow-varphi} for the flow map -- which can be written as the property $\delta\varphi(t_1,t_2,x):=\varphi(t_1+t_2,x)-\varphi(t_1,\Phi(t_2,x))=0$ for all $t_1,t_2\in\R$ and $x\in\T^d$. Lemma~\ref{lem:integrator} below gives an upper bound for $\delta\Phi(t_1,t_2,x)$ which is crucial for the proof of Theorem~\ref{theo:main}.
\begin{lemma}\label{lem:integrator}
There exists $C\in(0,\infty)$ such that for all $t_1,t_2\in\R$ and $x\in \T^d$, one has
\begin{equation}\label{eq:lem-integrator}
\dd\bigl(\Phi(t_1+t_2,x),\Phi(t_1,\Phi(t_2,x)\bigr)\le C\bigl(|t_1|t_2^2+t_1^2|t_2|\bigr)e^{C(|t_1|+|t_2|)}.
\end{equation}
\end{lemma}

\begin{proof}[Proof of Lemma~\ref{lem:integrator}]
Owing to the regularity conditions on $\Phi$ from Lemma~\ref{ass:integrator}, the mapping $\delta\Phi:\R^2\times\T^d\to \R^d$ is of class $\mathcal{C}^3$, and for all $t_1,t_2\in\R$ and $x\in\T^d$, one has
\begin{align*}
\partial_{t_1}\partial_{t_2}\delta\Phi(t_1,t_2,x)&=\partial_{t_1}\partial_{t_2}\bigl(\Phi(t_1+t_2,x)-\Phi(t_1,\Phi(t_2,x))\bigr)\\
&=\partial_{t_1}\bigl(\partial_t\Phi(t_1+t_2,x)-\partial_x\Phi(t_1,\Phi(t_2,x))\partial_t\Phi(t_2,x)\bigr)\\
&=\partial_t^2\Phi(t_1+t_2,x)-\partial_t\partial_x\Phi(t_1,\Phi(t_2,x))\partial_t\Phi(t_2,x).
\end{align*}
As a consequence of the three identities~\eqref{eq:deriveesPhi}, for all $x\in\T^d$, one has
\[
\partial_{t_1}\partial_{t_2}\delta\Phi(t_1=0,t_2=0,x)=0.
\]
In addition, as a consequence of the condition~\eqref{eq:ass_integrator-derivees} (Assumption~\ref{ass:integrator}), one has upper bounds
\[
\underset{x\in\T^d}\sup~|\partial_{t_1}^2\partial_{t_2}\delta\Phi(t_1,t_2,x)|+\underset{x\in\T^d}\sup~|\partial_{t_1}\partial_{t_2}^2e(t_1,t_2;x)|\le Ce^{C|t_1|+C|t_2|}.
\]
for all $t_1,t_2\in\R$. Therefore one obtains the inequality
\[
\underset{x\in\T^d}\sup~|\partial_{t_1}\partial_{t_2}\delta\Phi(t_1,t_2,x)|\le Ce^{C|t_1|+C|t_2|}(|t_1|+|t_2|)
\]
for all $t_1,t_2\in\R$.

Note also that, for all $x\in\T^d$, one has
\[
\delta\Phi(0,t_2,x)=\partial_{t_2}\delta\Phi(0,t_2,x)=0
\]
for all $t_2\in\R$ and
\[
\delta\Phi(t_1,0,x)=0
\]
for all $t_1\in\R$.

First, integrating with respect to the $t_1$ variable gives for all $t_1,t_2\in\R$
\[
|\partial_{t_2}\delta\Phi(t_1,t_2,x)|\le \int_{0}^{|t_1|}|\partial_{t_1}\partial_{t_2}\delta\Phi(t,t_2,x)|dt\le Ce^{C|t_1|+C|t_2|}(|t_1|+|t_2|)|t_1|.
\]
Second, integrating with respect to the $t_2$ variable gives for all $t_1,t_2\in\R$
\[
|\delta\Phi(t_1,t_2,x)|\le \int_{0}^{|t_2|}|\partial_{t_2}\delta\Phi(t_1,t,x)| dt\le Ce^{C|t_1|+C|t_2|}(|t_1|+|t_2|)|t_1||t_2|.
\]
Using the identity $(|t_1|+|t_2|)|t_1||t_2|=\bigl(|t_1|t_2^2+t_1^2|t_2|\bigr)$ then concludes the proof of Lemma~\ref{lem:integrator}.
\end{proof}

\begin{rem}
For the standard explicit Euler scheme, such that $\Phi^{\rm E}(t,x)=x+t\sigma(x)$, one has
\begin{align*}
\delta\Phi^{\rm E}(t_1,t_2,x)&=x+(t_1+t_2)\sigma(x)-\bigl(\Phi^{\rm E}(t_2,x)+t_1\sigma(\Phi^{\rm E}(t_2,x))\bigr)\\
&=t_1\bigl(\sigma(x)-\sigma(x+t_2\sigma(x))\bigr)
\end{align*}
and one only obtains the inequality
\[
\dd\bigl(\Phi^{\rm E}(t_1+t_2,x),\Phi^{\rm E}(t_1,\Phi^{\rm E}(t_2,x)\bigr)\le C|t_1||t_2|.
\]
The comparison of this inequality with~\eqref{eq:lem-integrator} illustrates why using a second-order integrator (satisfying~\eqref{eq:ass_integrator-order} from Assumption~\ref{ass:integrator}) is crucial in the proof of uniform strong error estimates.
\end{rem}

\subsection{Stong error estimate for the approximation of the Ornstein--Uhlenbeck component}\label{sec:aux-errorOU}

Let us provide a useful strong error estimate for the approximation of $m^\epsilon(t_n)$ by $m_n^{\epsilon,\Delta t}$. 

\begin{lemma}\label{lem:cv_m}
For all $p\in[1,\infty)$, there exists $C_p\in(0,\infty)$ such that for all $\epsilon\in(0,\epsilon_0)$, all $\Delta t\in(0,\Delta t_0)$ and all $n\in\N$, one has
\begin{equation}\label{eq:cv_m}
\|m_n^{\epsilon,\Delta t}-m^\epsilon(t_n)\|_p\le C_p\bigl(\frac{\sqrt{\Delta t}}{\epsilon}+\frac{1}{n}\bigr).
\end{equation}
\end{lemma}

Note that the error estimate~\eqref{eq:cv_m} is not uniform with respect to $\epsilon$. In addition, an additional error term $1/n$ appears, which is not small if $n=1$ for instance. In the proof of Theorem~\ref{theo:main} in Section~\ref{sec:proofmain}, Lemma~\ref{lem:cv_m} is used to obtain upper bounds for $\epsilon\|m_n^{\epsilon,\Delta t}-m^\epsilon(t_n)\|_p$, which are uniform with respect to $\epsilon$. The error terms $1/n$ and $1/n^2$ are summed for $n=1,\ldots,N$ and using
\[
\sum_{n=1}^{N}\frac1n\le C(T)|\log(\Delta t)|~,\quad \sum_{n=1}^{N}\frac{1}{n^2}\le C
\]
with $T=N\Delta t$ is sufficient to obtain the required strong error estimates.

Lemma~\ref{lem:cv_m} has been stated and proved in the PhD thesis~\cite{RR-thesis}. However, due to its importance for the proof of Theorem~\ref{theo:main}, a detailed proof is given below.

\begin{proof}[Proof of Lemma~\ref{lem:cv_m}]
It suffices to consider the case $p=2$, since the random variable $m_n^{\epsilon,\Delta t}-m^\epsilon(t_n)$ is Gaussian.

Using the identities
\begin{align*}
m^\epsilon(t_n)&=e^{-\frac{t_n}{\epsilon^2}}m_0^\epsilon+\frac{1}{\epsilon}\int_{0}^{t_n}e^{-\frac{t_n-s}{\epsilon^2}}d\beta(s)=e^{-\frac{t}{\epsilon^2}}m_0^\epsilon+\frac{1}{\epsilon}\sum_{\ell=0}^{n-1}\int_{t_\ell}^{t_{\ell+1}}e^{-\frac{t_n-s}{\epsilon^2}}d\beta(s)\\
m_n^{\epsilon,\Delta t}&=\frac{1}{(1+\frac{\Delta t}{\epsilon^2})^n}m_0^\epsilon+\frac{1}{\epsilon}\sum_{\ell=0}^{n-1}\frac{1}{(1+\frac{\Delta t}{\epsilon^2})^{n-\ell}}\bigl(\beta(t_{n+1})-\beta(t_n)\bigr),
\end{align*}
and It\^o's isometry formula, one obtains
\begin{align*}
\E[|m_n^{\epsilon,\Delta t}-m^\epsilon(t_n)|^2]&\le 3\Bigl(\frac{1}{(1+\frac{\Delta t}{\epsilon^2})^n}-e^{-\frac{n\Delta t}{\epsilon^2}}\Bigr)^2|m_0^\epsilon|^2\\
&+\frac{3\Delta t}{\epsilon^2}\sum_{\ell=0}^{n-1}\Bigl(\frac{1}{(1+\frac{\Delta t}{\epsilon^2})^{n-\ell}}-e^{-\frac{(n-\ell)\Delta t}{\epsilon^2}}\Bigr)^2\\
&+\frac{3}{\epsilon^2}\sum_{\ell=0}^{n-1}\int_{t_\ell}^{t_{\ell+1}}\Bigl(e^{-\frac{t_n-t_{\ell}}{\epsilon^2}}-e^{-\frac{t_n-s}{\epsilon^2}}\Bigr)^2ds.
\end{align*}
Note that one has the inequality
\[
\underset{n\in\N}\sup~\underset{z\in(0,\infty)}\sup~n\big|\frac{1}{(1+z)^n}-e^{-nz}\big|<\infty.
\]
Indeed, for all $n\in\N$ the maximum of the function
\[
z\in[0,\infty)\mapsto \frac{1}{(1+z)^n}-e^{-nz}\in[0,\infty)
\]
is attained for a real number $z=z_n$ satisfying $e^{-nz_n}=\frac{1}{(1+z_n)^{n+1}}$: as a consequence
\[
\underset{z\in(0,\infty)}\sup~\bigl(\frac{1}{(1+z)^n}-e^{-nz}\bigr)=\frac{1}{(1+z_n)^n}-e^{-nz_n}=\frac{z_n}{(1+z_n)^{n+1}}\le \frac{z_n}{(n+1)z_n}\le \frac{1}{n+1}.
\]
As a consequence of the inequality above, one obtains upper bounds for the first and the second error terms above:
\[
\Bigl(\frac{1}{(1+\frac{\Delta t}{\epsilon^2})^n}-e^{-\frac{n\Delta t}{\epsilon^2}}\Bigr)^2|m_0^\epsilon|^2\le \frac{C}{n^2},
\]
using Assumption~\ref{ass:init}, and
\begin{align*}
\frac{\Delta t}{\epsilon^2}\sum_{\ell=0}^{n-1}\Bigl(\frac{1}{(1+\frac{\Delta t}{\epsilon^2})^{n-\ell}}-e^{-\frac{(n-\ell)\Delta t}{\epsilon^2}}\Bigr)^2&\le C\frac{\Delta t}{\epsilon^2}\sum_{\ell=0}^{n-1}\frac{1}{(n-\ell)^2}\le C\frac{\Delta t}{\epsilon^2}.
\end{align*}
The third error term is treated as follows: using the inequality
\[
\underset{z\in(0,\infty)}\sup~\frac{(e^{-z}-1)^2}{z}<\infty,
\]
one obtains
\begin{align*}
\frac{1}{\epsilon^2}\sum_{\ell=0}^{n-1}\int_{t_\ell}^{t_{\ell+1}}\Bigl(e^{-\frac{t_n-t_{\ell}}{\epsilon^2}}-e^{-\frac{t_n-s}{\epsilon^2}}\Bigr)^2ds&=\frac{1}{\epsilon^2}\sum_{\ell=0}^{n-1}\int_{t_\ell}^{t_{\ell+1}}\Bigl(e^{-\frac{s-t_{\ell}}{\epsilon^2}}-1\Bigr)^2 e^{-\frac{t_n-s}{\epsilon^2}}ds\\
&\le \frac{C}{\epsilon^2}\frac{\Delta t}{\epsilon^2}\int_{0}^{t_n}e^{-\frac{t_n-s}{\epsilon^2}}ds\\
&\le \frac{C\Delta t}{\epsilon^2}.
\end{align*}
Gathering the estimates then concludes the proof of Lemma~\ref{lem:cv_m}.
\end{proof}

\subsection{Properties of the Ornstein--Uhlenbeck component}\label{sec:aux-boundsOU}

This section is devoted to the proof of Lemma~\ref{lem:OU} below, which states a series of results concerning $m^\epsilon(t)$ and $\zeta^\epsilon(t)$, see~\eqref{eq:mepsilon} and~\eqref{eq:zeta} respectively.

\begin{lemma}\label{lem:OU}\hspace{1cm}

$\bullet$ 
One has moment bounds for $m^\epsilon(t)$, uniformly with respect to $\epsilon\in(0,\epsilon_0)$ and $t\ge 0$: for all $p\in[1,\infty)$, one has
\begin{equation}\label{eq:bound-mepsilon}
\underset{\epsilon\in(0,\epsilon_0)}\sup~\underset{t\ge 0}\sup~\|m^\epsilon(t)\|_p<\infty.
\end{equation}
In addition, one has exponential moment bounds for $m^\epsilon(t)$, uniformly with respect to $\epsilon\in(0,\epsilon_0)$ and $t\ge 0$: for all $q\in(0,\infty)$, one has
\begin{equation}\label{eq:expomoments-m}
\underset{\epsilon\in(0,\epsilon_0)}\sup~\underset{t\ge 0}\sup~\E[e^{q|m^\epsilon(t)|}]<\infty.
\end{equation}

$\bullet$ 
For all $t\ge 0$, $\zeta^\epsilon(t)$ converges to $\beta(t)$ when $\epsilon\to 0$, in the following sense: for all $p\in[1,\infty)$, there exists $C_p\in(0,\infty)$ such that for all $\epsilon\in(0,\epsilon_0)$ one has
\begin{equation}\label{eq:cvzeta}
\underset{t\ge 0}\sup~\|\zeta^\epsilon(t)-\beta(t)\|_p\le C_p\epsilon.
\end{equation}

$\bullet$ 
One has exponential moment bounds for $\zeta^\epsilon(t)$ and $\beta(t)$, uniformly with respect to $\epsilon\in(0,\epsilon_0)$ and to $t\in[0,T]$: for all $T\in(0,\infty)$ and all $q\in(0,\infty)$, one has
\begin{equation}\label{eq:expomoments-zeta-beta}
\underset{\epsilon\in(0,\epsilon_0)}\sup~\underset{0\le t\le T}\sup~\E[e^{q|\zeta^\epsilon(T)|}]+\underset{0\le t\le T}\sup~\E[e^{q|\beta(T)|}]<\infty.
\end{equation}

$\bullet$ 
One has bounds on the increments of $\zeta^\epsilon$, uniformly with respect to $\epsilon\in(0,\epsilon_0)$: for all $p\in[1,\infty)$, there exists $C_p\in(0,\infty)$ such that for all $t_1,t_2\ge 0$ one has
\begin{equation}\label{eq:incrementszeta}
\underset{\epsilon\in(0,\epsilon_0)}\sup~\E[|\zeta^\epsilon(t_2)-\zeta^\epsilon(t_1)|^{2p}]\le C_p|t_2-t_1|^p.
\end{equation}
\end{lemma}

\begin{proof}[Proof of Lemma~\ref{lem:OU}]

$\bullet$ {\it Proof of the inequalities~\eqref{eq:bound-mepsilon} and~\eqref{eq:expomoments-m}.}

Since $m^\epsilon(t)$ is a Gaussian random variable, it suffices to consider the case $p=2$. For all $\epsilon\in(0,\epsilon_0)$ and $t\ge 0$, one has
\[
m^\epsilon(t)=e^{-\frac{t}{\epsilon^2}}m_0^\epsilon+\frac{1}{\epsilon}\int_{0}^{t}e^{-\frac{t-s}{\epsilon^2}}d\beta(s),
\]
and applying It\^o's isometry formula yields
\[
\E[|m^\epsilon(t)|^2]=e^{-2\frac{t}{\epsilon^2}}|m_0^\epsilon|^2+\frac{1}{\epsilon^2}\int_{0}^{t}e^{-2\frac{t-s}{\epsilon^2}}ds\le \underset{\varepsilon\in(0,\epsilon_0)}\sup~|m_0^\varepsilon|^2+\frac{1}{2}.
\]
Using Assumption~\ref{ass:init} then yields~\eqref{eq:bound-mepsilon}.

The exponential moment bounds~\eqref{eq:expomoments-m} are then a straightforward consequence of the uniform boundedness with respect to $\epsilon\in(0,\epsilon_0)$ and $t\ge 0$ of the mean and of the variance of the Gaussian random variable $m^\epsilon(t)$.

$\bullet$ {\it Proof of the inequality~\eqref{eq:cvzeta}.}

Observe that for all $t\ge 0$ and $\epsilon\in(0,\epsilon_0)$, one has
\[
m^\epsilon(t)-m^\epsilon(0)=-\frac{1}{\epsilon^2}\int_{0}^{t}m^\epsilon(s)ds+\frac{1}{\epsilon}\beta(t),
\]
therefore one has the identity
\[
\zeta^\epsilon(t)=\frac{1}{\epsilon}\int_{0}^{t}m^\epsilon(s)ds=\beta(t)+\epsilon(m_0^\epsilon-m^\epsilon(t)).
\]
Using the inequality~\eqref{eq:bound-mepsilon}, one then obtains the error estimate
\[
\E[|\zeta^\epsilon(t)-\beta(t)|^2]\le 2\epsilon^2\underset{\varepsilon\in(0,\epsilon_0)}\sup~\underset{t\ge 0}\sup~\E[|m^\varepsilon(t)|^2]\le C\epsilon^2.
\]
The constant $C\in(0,\infty)$ does not depend on $\epsilon\in(0,\epsilon_0$ or $t\ge 0$. This gives~\eqref{eq:cvzeta} with $p=2$. Since $\zeta^\epsilon(T)-\beta(T)$ is a Gaussian random variable, this also yields~\eqref{eq:cvzeta} for arbitrary $p\in[1,\infty)$.

$\bullet$ {\it Proof of the inequality~\eqref{eq:expomoments-zeta-beta}.}

First, $\beta(t)$ is a centered Gaussian random variable with variance $\E[|\beta(t)|^2]=t\le T$. As a consequence, one has
\[
\underset{0\le t\le T}\sup~\E[e^{\frac{1}{4T}|\beta(t)|^2}]<\infty.
\]
Let $q\in(0,\infty)$. Using Young's inequality, one has $4q|\beta(t)|\le \frac{1}{4T}|\beta(t)|^2+16Tq^2$, therefore one obtains the exponential moment bounds 
\[
\underset{0\le t\le T}\sup~\E[e^{q|\beta(T)|}]<\infty.
\]
Second, note that
\[
\E[e^{q|\zeta^\epsilon(t)|}]\le \E[e^{q|\beta(t)|+q|\zeta^\epsilon(t)-\beta(T)|}]\le \bigl(\E[e^{2q|\beta(t)|}]\bigr)^{\frac12} \bigl(\E[e^{2q|\zeta^\epsilon(t)-\beta(t)|}]\bigr)^{\frac12},
\]
using the Cauchy--Schwarz inequality. It suffices to deal with the second factor in the right-hand side above, the first factor being upper bounded using the estimate proved above. Owing to the inequality~\eqref{eq:cvzeta}, the Gaussian random variable $\zeta^\epsilon(t)-\beta(t)$ has a mean and a variance which are bounded uniformly with respect to $\epsilon\in(0,\epsilon_0)$ and $t\in[0,\infty)$. As a consequence, there exists $c\in(0,\infty)$ such that
\[
\underset{\epsilon\in(0,\epsilon_0)}\sup~\underset{t\ge 0}\sup~\E[e^{c|\zeta^\epsilon(t)-\beta(t)|^2}]<\infty.
\]
Using Young's inequality gives
\[
2q|\zeta^\epsilon(t)-\beta(t)|\le c|\zeta^\epsilon(t)-\beta(t)|^2+\frac{q^2}{c},
\]
and the conclusion of the proof of the inequality~\eqref{eq:expomoments-zeta-beta} is then straightforward.

$\bullet$ {\it Proof of the inequality~\eqref{eq:incrementszeta}.}

Since the random variable
\[
\zeta^\epsilon(t_2)-\zeta^\epsilon(t_1)=\frac{1}{\epsilon}\int_{t_1}^{t_2}m^\epsilon(s)ds
\]
is Gaussian, it suffices to consider the case $p=1$.

Without loss of generality, assume that $t_1\le t_2$.

First, note that the mean of $\zeta^\epsilon(t_2)-\zeta^\epsilon(t_1)$ satisfies
\begin{align*}
\big|\E[\zeta^\epsilon(t_2)-\zeta^\epsilon(t_1)]\big|&=\frac{1}{\epsilon}\big|\int_{t_1}^{t_2}\E[m^\epsilon(s)]ds\big|=\frac{1}{\epsilon}\int_{t_1}^{t_2}e^{-\frac{s}{\epsilon^2}}ds |m_0^\epsilon|\le \epsilon\bigl(e^{-\frac{t_1}{\epsilon^2}}-e^{-\frac{t_2}{\epsilon^2}}\bigr)|m_0^\epsilon|\le C(t_2-t_1)^{\frac12},
\end{align*}
where $C$ does not depend on $\epsilon$, using the inequality
\[
\underset{z_1,z_2\ge 0}\sup~\frac{|e^{-z_2}-e^{-z_1}|}{|z_2-z_1|^{\frac12}}<\infty
\]
and Assumption~\ref{ass:init}.

Second, the variance of $\zeta^\epsilon(t_2)-\zeta^\epsilon(t_1)$ satisfies
\begin{align*}
\E[\big|&\zeta^\epsilon(t_2)-\zeta^\epsilon(t_1)-\E[\zeta^\epsilon(t_2)-\zeta^\epsilon(t_1)]\big|^2]=\frac{1}{\epsilon^2}\E[\big|\int_{t_1}^{t_2}\bigl(m^\epsilon(s)-\E[m^\epsilon(s)]\bigr)ds\big|^2\\
&=\frac{1}{\epsilon^2}\int_{t_1}^{t_2}\int_{t_1}^{t_2}\E\Bigl[\frac{1}{\epsilon}\int_{0}^{s_1}e^{-\frac{(s_1-s)}{\epsilon^2}}d\beta(s)\frac{1}{\epsilon}\int_{0}^{s_2}e^{-\frac{(s_2-s)}{\epsilon^2}}d\beta(s)\Bigr]ds_1ds_2\\
&=\frac{2}{\epsilon^2}\int_{t_1}^{t_2}\int_{s_1}^{t_2}\frac{1}{\epsilon^2}\int_{0}^{s_1}e^{-\frac{s_2-s_1}{\epsilon^2}}e^{-2\frac{(s_1-s)}{\epsilon^2}}ds ds_2ds_1\\
&\le \frac{1}{\epsilon^2}\int_{t_1}^{t_2}\int_{s_1}^{t_2}e^{-\frac{s_2-s_1}{\epsilon^2}}(1-e^{-\frac{2s_1}{\epsilon^2}})ds_2ds_1\\
&\le t_2-t_1.
\end{align*}
Combining the estimates concludes the proof of the inequality~\eqref{eq:incrementszeta}.
\end{proof}

\subsection{Properties of the discretized Ornstein--Uhlenbeck component}\label{sec:aux-boundsOU-discrete}

This section is devoted to the proof of Lemma~\ref{lem:OUdiscrete} below, which is a variant of Lemma~\ref{lem:OU} from Section~\ref{sec:aux-boundsOU} above, which states a series of results concerning $m_n^{\epsilon,\Delta t}$.

\begin{lemma}\label{lem:OUdiscrete}\hspace{1cm}

$\bullet$ 
One has moment bounds for $m_n^{\epsilon,\Delta t}$, uniformly with respect to $\Delta t\in(0,\Delta t_0)$, $\epsilon\in(0,\epsilon_0)$ and $n\in\N$: for all $p\in[1,\infty)$, one has
\begin{equation}\label{eq:bound-mespsilonDeltat}
\underset{\epsilon\in(0,\epsilon_0)}\sup~\underset{\Delta t\in(0,\Delta t_0)}\sup~\underset{n\in\N}\sup~\|m_n^{\epsilon,\Delta t}\|_p<\infty.
\end{equation}

In addition, one has exponential moment bounds for $m_n^{\epsilon,\Delta t}$, uniformly with respect to $\Delta t\in(0,\Delta t_0$, $\epsilon\in(0,\epsilon_0)$ and $n\in\N$: for all $q\in(0,\infty)$, one has
\begin{equation}\label{eq:expomoments-m-discrete}
\underset{\epsilon\in(0,\epsilon_0)}\sup~\underset{\Delta t\in(0,\Delta t_0)}\sup~\underset{n\in\N}\sup~\E[e^{q|m_n^{\epsilon,\Delta t}|}]<\infty.
\end{equation}

$\bullet$ 
For all $n\ge 0$, $\frac{\Delta tm_{n+1}^{\epsilon,\Delta t}}{\epsilon}$ converges to $\Delta \beta_n=\beta(t_{n+1})-\beta(t_n)$ when $\epsilon\to 0$, in the following sense: for all $p\in[1,\infty)$, there exists $C_p\in(0,\infty)$ such that for all $\epsilon\in(0,\epsilon_0)$ one has
\begin{equation}\label{eq:cvdiscrete}
\underset{\Delta t\in(0,\Delta t_0)}\sup~\underset{n\ge 0}\sup~\|\frac{\Delta tm_{n+1}^{\epsilon,\Delta t}}{\epsilon}-\Delta \beta_n\|_p\le C_p\epsilon.
\end{equation}

$\bullet$ 
One has exponential moment bounds for $\frac{\Delta tm_{n+1}^{\epsilon,\Delta t}}{\epsilon}$, uniformly with respect to $\Delta t\in(0,\Delta t_0)$, $\epsilon\in(0,\epsilon_0)$ and to $n\ge 0$: for all $T\in(0,\infty)$ and all $q\in(0,\infty)$, one has
\begin{equation}\label{eq:expomoments-discrete}
\underset{\epsilon\in(0,\epsilon_0)}\sup~\underset{\Delta t\in(0,\Delta t_0)}\sup~\underset{n\ge 0}\sup~\E[e^{q|\frac{\Delta tm_{n+1}^{\epsilon,\Delta t}}{\epsilon}|}]<\infty.
\end{equation}

$\bullet$
One has an error estimate for $\frac{\Delta tm_{n+1}^{\epsilon,\Delta t}}{\epsilon}$, uniformly with respect to $\epsilon\in(0,\epsilon_0)$: for all $p\in[1,\infty)$, there exists $C_p\in(0,\infty)$ such that one has
\begin{equation}\label{eq:incrementsdiscrete}
\underset{\epsilon\in(0,\Delta t_0)}\sup~\underset{\Delta t\in(0,\Delta t_0)}\sup~\underset{n\ge 0}\sup~\frac{1}{\Delta t^p}\E[|\frac{\Delta tm_{n+1}^{\epsilon,\Delta t}}{\epsilon}|^{2p}]<\infty.
\end{equation}
\end{lemma}

\begin{proof}[Proof of Lemma~\ref{lem:OUdiscrete}]
$\bullet$ {\it Proof of the inequalities~\eqref{eq:bound-mespsilonDeltat} and~\eqref{eq:expomoments-m-discrete}}.

Since $m_n^{\epsilon,\Delta t}$ is a Gaussian random variable, it suffices to consider the case $p=2$. For all $\epsilon\in(0,\epsilon_0)$, $\Delta t\in(0,\Delta t_0)$ and $n\ge 0$, one has
\[
m_n^{\epsilon,\Delta t}=\frac{1}{(1+\frac{\Delta t}{\epsilon^2})^n}m_0^\epsilon+\frac{1}{\epsilon}\sum_{\ell}^{n-1}\frac{1}{(1+\frac{\Delta t}{\epsilon^2})^{n-\ell}}\Delta\beta_\ell.
\]
Since the Gaussian random variables $\bigl(\Delta\beta_\ell\bigr)_{\ell\ge 0}$ are centered and independent, with variance $\Delta t$, one obtains
\[
\E[|m_n^{\epsilon,\Delta t}|^2]=\frac{1}{(1+\frac{\Delta t}{\epsilon^2})^{2n}}|m_0^\epsilon|^2+\frac{\Delta t}{\epsilon^2}\sum_{\ell=0}^{n-1}\frac{1}{(1+\frac{\Delta t}{\epsilon^2})^{2(n-\ell)}}\le |m_0^\epsilon|^2+\frac{1}{2+\frac{\Delta t}{\epsilon^2}}\le \underset{\varepsilon\in(0,\epsilon_0)}\sup~|m_0^\varepsilon|^2+\frac12.
\]
Using Assumption~\ref{ass:init} then yields~\eqref{eq:bound-mespsilonDeltat}.

The exponential moment bounds~\eqref{eq:expomoments-m-discrete} are then a straightforward consequence of the uniform boundedness with respect to $\epsilon\in(0,\epsilon_0)$ and $t\ge 0$ of the mean and of the variance of the Gaussian random variable $m_n^{\epsilon,\Delta t}$.

$\bullet$ {\it Proof of the inequality~\eqref{eq:cvdiscrete}.}

Since $\frac{\Delta tm_{n+1}^{\epsilon,\Delta t}}{\epsilon}-\sqrt{\Delta t}\gamma_n$ is a Gaussian random variable, it suffices to consider the case $p=2$. By the definition of the scheme, one has the equality
\[
\frac{\Delta tm_{n+1}^{\epsilon,\Delta t}}{\epsilon}=\Delta\beta_n+\epsilon(m_n^{\epsilon,\Delta t}-m_{n+1}^{\epsilon,\Delta t}).
\]
Using the inequality~\eqref{eq:bound-mespsilonDeltat}, one then obtains~\eqref{eq:cvdiscrete}.

$\bullet$ {\it Proof of the inequality~\eqref{eq:expomoments-discrete}.}

As a consequence of the inequality~\eqref{eq:cvdiscrete}, the mean and variance of the Gaussian random variable $\frac{\Delta tm_{n+1}^{\epsilon,\Delta t}}{\epsilon}$ are bounded uniformly with respect to $\epsilon\in(0,\epsilon_0)$, $\Delta t\in(0,\Delta t_0)$ and $n\ge 0$, there exists $c\in(0,\infty)$ such that
\[
\underset{\epsilon\in(0,\epsilon_0)}\sup~\underset{\Delta t\in(0,\Delta t_0)}\sup~\underset{n\in\ge 0}\sup~\E[e^{c|\frac{\Delta tm_{n+1}^{\epsilon,\Delta t}}{\epsilon}|^2}]<\infty.
\]
Using Young's inequality then concludes the proof of the inequality~\eqref{eq:expomoments-discrete}. The details are similar to those used in the proof of Lemma~\ref{lem:OU} and are omitted.

$\bullet$ {\it Proof of the inequality~\eqref{eq:incrementsdiscrete}.}

Since $\frac{\Delta tm_{n+1}^{\epsilon,\Delta t}}{\epsilon}$ is a Gaussian random variable, it suffices to consider the case $p=1$.

Using the identity
\[
m_n^{\epsilon,\Delta t}=\frac{1}{(1+\frac{\Delta t}{\epsilon^2})^n}m_0^\epsilon+\frac{}{\epsilon}\sum_{\ell=0}^{n-1}\frac{1}{(1+\frac{\Delta t}{\epsilon^2})^{n-\ell}}\Delta\beta_\ell,
\]
and the equality
\[
\frac{\Delta t}{\epsilon^2}\sum_{\ell=0}^{n-1}\frac{1}{(1+\frac{\Delta t}{\epsilon^2})^{2(n-\ell)}}=\frac{1}{2+\frac{\Delta t}{\epsilon^2}},
\]
one obtains
\begin{align*}
\E[|\frac{\Delta tm_{n+1}^{\epsilon,\Delta t}}{\epsilon}|^2]&\le\Delta t\frac{\frac{\Delta t}{\epsilon^2}}{(1+\frac{\Delta t}{\epsilon})^{2(n+1)}}|m_0^\epsilon|^2+\Delta t\frac{\frac{\Delta t}{\epsilon^2}}{2+\frac{\Delta t}{\epsilon^2}}\\
&\le \Delta t|m_0^\epsilon|^2+\Delta t.
\end{align*}
Using Assumption~\ref{ass:init}, this concludes the proof of the inequality~\eqref{eq:incrementsdiscrete}.
\end{proof}

\section{Proofs of the results from Section~\ref{sec:results-epsilon}}\label{sec:proofs}

This short section is devoted to giving detailed proofs of Proposition~\ref{propo:cvSDE} and~\ref{propo:cvScheme}, concerning the asymptotic behavior of $X^{\epsilon}(t)$ and $X_n^{\epsilon,\Delta t}$ when $\epsilon\to 0$ respectively. The proofs are straightforward consequences of the auxiliary results studied in Section~\ref{sec:aux}.

\subsection{Proof of Proposition~\ref{propo:cvSDE}}

\begin{proof}[Proof of Proposition~\ref{propo:cvSDE}]
Let $T\in(0,\infty)$ and $p\in[1,\infty)$. For all $t\in[0,T]$, the random variables $X^\epsilon(t)$ and $X(t)$ are expressed in terms of the flow map $\varphi$ and of the Gaussian random variables $\zeta^{\epsilon}(t)$ and $\beta(t)$ using~\eqref{eq:SDE-solution} and~\eqref{eq:limitingSDE-solution} respectively. Applying Lemma~\ref{lem:flow}, one obtains
\begin{align*}
\dd_p(X^{\epsilon}(t),X(t))&
=\dd_p\bigl(\varphi(\zeta^\epsilon(t),x_0^\epsilon),\varphi(\beta(t),x_0)\bigr)\\
&\le \|\zeta^\epsilon(t)-\beta(t)\|_p+\|e^{C|\zeta^\epsilon(t)|+C|\beta(t)|}\|_p\dd(x_0^\epsilon,x_0^0).
\end{align*}
Using the inequality~\eqref{eq:cvzeta} and the exponential moment bounds~\eqref{eq:expomoments-zeta-beta} from Lemma~\ref{lem:OU} yields the inequality
\[
\underset{0\le t\le T}\sup~\dd_p(X^{\epsilon}(t),X(t))\le C_p(T)\bigl(\epsilon+\dd(x_0^\epsilon,x_0^0)\bigr).
\]
Using Assumption~\ref{ass:init} then concludes the proof of Proposition~\ref{propo:cvSDE}.
\end{proof}

\subsection{Proof of Proposition~\ref{propo:cvScheme}}

\begin{proof}[Proof of Proposition~\ref{propo:cvScheme}]
Using the definitions~\eqref{eq:scheme} and~\eqref{eq:limitingscheme} of the schemes, for all $n\in\{0,\ldots,N-1\}$, one has
\begin{align*}
\dd(X_{n+1}^{\epsilon,\Delta t},X_{n+1}^{0,\Delta t})&=\dd\bigl(\Phi(\frac{\Delta t m_{n+1}^\epsilon}{\epsilon},X_n^{\epsilon,\Delta t}),\Phi(\Delta\beta_n,X_n^{0,\Delta t})\bigr)\\
&\le Ce^{C|\frac{\Delta tm_{n+1}^\epsilon}{\epsilon}|+C|\Delta\beta_n|}\bigl(|\frac{\Delta tm_{n+1}^\epsilon}{\epsilon}-\Delta\beta_n|+\dd(X_n^{\epsilon,\Delta t},X_n^{0,\Delta t})\bigr),
\end{align*}
using the inequality~\eqref{eq:ass_integrator-Lip} from Assumption~\ref{ass:integrator}. Using H\"older's inequality, the exponential moment bounds~\eqref{eq:expomoments-discrete} and the inequality~\eqref{eq:cvdiscrete} from Lemma~\ref{lem:OUdiscrete}, there exists $C_p\in(0,\infty)$ such that for all $n\in\{0,\ldots,\}$ one has
\[
\dd_p(X_{n+1}^{\epsilon,\Delta t},X_{n+1}^{0,\Delta t})\le C_p\bigl(\epsilon+\dd_{2p}(X_{n}^{\epsilon,\Delta t},X_{n}^{0,\Delta t})\bigr).
\]
Owing to Assumption~\ref{ass:init}, one has
\[
\dd_p(X_0^{\epsilon,\Delta t},X_0^{0,\Delta t})=\dd(x_0^\epsilon,x_0^0)\underset{\epsilon\to 0}\to 0
\]
and it is then straightforward to check recursively that for all $p\in[1,\infty)$ and $n\in\{0,\ldots,N\}$ one has
\[
\dd_p(X_{n+1}^{\epsilon,\Delta t},X_{n}^{0,\Delta t})\underset{\epsilon\to 0}\to 0.
\]
This concludes the proof of Proposition~\ref{propo:cvScheme}.
\end{proof}

\section{Proof of the main result}\label{sec:proofmain}

This section is devoted to the proof of Theorem~\ref{theo:main}.

\begin{proof}[Proof of Theorem~\ref{theo:main}]
Let us introduce an auxiliary process $\bigl(Y_n^{\epsilon,\Delta t}\bigr)_{0\le n\le N}$, defined by
\begin{equation}\label{eq:auxY}
Y_n^{\epsilon,\Delta t}=\Phi\bigl(\epsilon(m_n^\epsilon-m^\epsilon(t_n)),X_n^{\epsilon,\Delta t}\bigr),
\end{equation}
for all $n\in\{0,\ldots,N\}$.

The error is then decomposed as follows: for all $p\in[1,\infty)$, one has
\begin{equation}\label{eq:errordecomp}
\dd_p(X_N^{\epsilon,\Delta t},X^\epsilon(t_N))\le \dd_p(X_N^{\epsilon,\Delta t},Y_N^{\epsilon,\Delta t})+\dd_p(Y_N^{\epsilon,\Delta t},X^\epsilon(t_N)).
\end{equation}

The first error term $\dd_p(X_N^{\epsilon,\Delta t},Y_N^{\epsilon,\Delta t})$ in the right-hand side of~\eqref{eq:errordecomp} is treated as follows: using Lemma~\ref{lem:flow} and the inequality~\eqref{eq:ass_integrator-order} from Assumption~\ref{ass:integrator}, one obtains
\begin{align*}
\dd_p(X_N^{\epsilon,\Delta t},Y_N^{\epsilon,\Delta t})&=\dd_p\bigl(X_N^{\epsilon,\Delta t},\Phi\bigl(\epsilon(m_N^{\epsilon,\Delta t},m^\epsilon(t_N)),X_N^{\epsilon,\Delta t}\bigr)\bigr)\\
&\le \dd_p\bigl(\varphi(0,X_N^{\epsilon,\Delta t}),\varphi\bigl(\epsilon(m_N^{\epsilon,\Delta t}-m^\epsilon(t_N)),X_N^{\epsilon,\Delta t}\bigr)\bigr)\\
&+\dd_p\bigl(\varphi\bigl(\epsilon(m_N^{\epsilon,\Delta t}-m^\epsilon(t_N)),X_N^{\epsilon,\Delta t}\bigr),\Phi\bigl(\epsilon(m_n^{\epsilon,\Delta t}-m^\epsilon(t_N)),X_N^{\epsilon,\Delta t}\bigr)\bigr)\\
&\le C\epsilon\|m_N^{\epsilon,\Delta t}-m^\epsilon(t_N)\|_p+C\epsilon^3\|(m_N^{\epsilon,\Delta t}-m^\epsilon(t_N))^3 e^{C\epsilon|m_N^{\epsilon,\Delta t}-m^\epsilon(t_N)|}\|_p.
\end{align*}
Using H\"older's inequality, the exponential moment bounds~\eqref{eq:expomoments-m} (Lemma~\ref{lem:OU}) and~\eqref{eq:expomoments-m-discrete} (Lemma~\ref{lem:OUdiscrete}), and the inequality~\eqref{eq:cv_m} from Lemma~\ref{lem:cv_m}, one obtains
\begin{align*}
\dd_p(X_N^{\epsilon,\Delta t},Y_N^{\epsilon,\Delta t})&\le C_p\bigl(\|\epsilon(m_N^{\epsilon,\Delta t}-m^\epsilon(t_N))\|_{p}+\|\epsilon(m_N^{\epsilon,\Delta t}-m^\epsilon(t_N))\|_{4p}^3\bigr)\\
&\le C_p\epsilon\bigl(\frac{\sqrt{\Delta t}}{\epsilon}+\frac{1}{N}\bigr)+C_p\epsilon^3\bigl(\frac{\sqrt{\Delta t}}{\epsilon}+\frac{1}{N}\bigr)^3\\
&\le C_p(T)\Delta t^{\frac12},
\end{align*}
using the inequality $1/N=\Delta t/T\le C(T)\Delta t^{1/2}$ in the last step.

It remains to study the second error term $\dd_p(Y_N^{\epsilon,\Delta t},X^\epsilon(t_N))$ in the right-hand side of~\eqref{eq:errordecomp}. The strategy is based on a telescoping sum argument: using the expression~\eqref{eq:SDE-solution} for $X^\epsilon(t_N)=X^\epsilon(T)$ and the equalities $Y_0^{\epsilon,\Delta t}=X_0^{\epsilon,\Delta t}=X^\epsilon(0)$, one has
\begin{align*}
\dd\bigl(Y_N^{\epsilon,\Delta t}&,X^\epsilon(t_N)\bigr)=\dd\bigl(\varphi(0,Y_N^{\epsilon,\Delta t}),\varphi(\zeta^\epsilon(t_N),Y_0^{\epsilon,\Delta t})\bigr)\\
&\le\sum_{n=0}^{N-1}\dd\Bigl(\varphi(\zeta^\epsilon(t_N)-\zeta^\epsilon(t_{n+1}),Y_{n+1}^{\epsilon,\Delta t}),\varphi(\zeta^\epsilon(t_N)-\zeta^{\epsilon}(t_n),Y_{n}^{\epsilon,\Delta t})\Bigr)\\
&\le\sum_{n=0}^{N-1}\dd\Bigl(\varphi(\zeta^\epsilon(t_N)-\zeta^\epsilon(t_{n+1}),Y_{n+1}^{\epsilon,\Delta t}),\varphi\bigl((\zeta^\epsilon(t_N)-\zeta^{\epsilon}(t_{n+1}),\varphi(\zeta^\epsilon(t_{n+1})-\zeta^\epsilon(t_n),Y_{n}^{\epsilon,\Delta t})\bigr)\Bigr),
\end{align*}
where the last inequality is a consequence of the flow property~\eqref{eq:flow-varphi} of the map $\varphi$.

Using the Lipschitz continuity property of the mapping $\varphi(\zeta^\epsilon(t_N)-\zeta^\epsilon(t_{n+1}),\cdot)$ given in Lemma~\ref{lem:flow}, then H\"older's inequality and the exponential moment bounds~\eqref{eq:expomoments-zeta-beta} and~\eqref{eq:expomoments-discrete}, one obtains the inequality
\begin{equation}\label{eq:errorbound}
\begin{aligned}
\dd_p(Y_N^{\epsilon,\Delta t},X^\epsilon(t_N))&\le C\sum_{n=0}^{N-1}\|\dd\bigl(Y_{n+1}^{\epsilon,\Delta t},\varphi(\zeta^\epsilon(t_{n+1})-\zeta^\epsilon(t_n),Y_{n}^{\epsilon,\Delta t})\bigr)e^{C|\zeta^\epsilon(t_N)-\zeta^\epsilon(t_{n+1})|}\|_p\\
&\le C_p(T)\sum_{n=0}^{N-1}\dd_{2p}\bigl(Y_{n+1}^{\epsilon,\Delta t},\varphi(\zeta^\epsilon(t_{n+1})-\zeta^\epsilon(t_n),Y_{n}^{\epsilon,\Delta t})\bigr).
\end{aligned}
\end{equation}
We claim that the following identity holds: for all $n\in\{0,\ldots,N-1\}$, one has
\begin{equation}\label{eq:claim}
Y_{n+1}^{\epsilon,\Delta t}-\Phi(\zeta^\epsilon(t_{n+1})-\zeta^\epsilon(t_n),Y_{n}^{\epsilon,\Delta t})=R_{n,1}^{\epsilon,\Delta t}+R_{n,2}^{\epsilon,\Delta t}
\end{equation}
with the error terms $R_{n,1}^{\epsilon,\Delta t}$ and $R_{n,2}^{\epsilon,\Delta t}$ defined by
\begin{align*}
R_{n,1}^{\epsilon,\Delta t}&=-\delta\Phi\Bigl(\epsilon(m_{n+1}^{\epsilon,\Delta t}-m^\epsilon(t_{n+1})),\frac{\Delta tm_{n+1}^{\epsilon,\Delta t}}{\epsilon},X_n^\epsilon\Bigr)\\
R_{n,2}^{\epsilon,\Delta t}&=\delta\Phi\Bigl(\zeta^\epsilon(t_{n+1})-\zeta^\epsilon(t_n),\epsilon(m_n^{\epsilon,\Delta t}-m^\epsilon(t_n)),X_n^\epsilon\Bigr).
\end{align*}
using the auxiliary function $\delta\Phi$ defined by~\eqref{eq:deltaPhi}.

The proof of the claim~\eqref{eq:claim} is performed in two steps. First, using the definition~\eqref{eq:auxY} of the auxiliary random variable $Y_{n+1}^{\epsilon,\Delta t}$ and the definition~\eqref{eq:scheme} of the numerical scheme, one has
\begin{align*}
Y_{n+1}^{\epsilon,\Delta t}&=\Phi(\epsilon(m_{n+1}^{\epsilon,\Delta t}-m^\epsilon(t_{n+1})),X_{n+1}^{\epsilon,\Delta t})\\
&=\Phi(\epsilon(m_{n+1}^{\epsilon,\Delta t}-m^\epsilon(t_{n+1})),\Phi(\frac{\Delta tm_{n+1}^{\epsilon,\Delta t}}{\epsilon},X_n^\epsilon))\\
&=\Phi(\epsilon(m_{n+1}^\epsilon-m^\epsilon(t_{n+1}))+\frac{\Delta tm_{n+1}^{\epsilon,\Delta t}}{\epsilon},X_n^\epsilon)+R_{n,1}^{\epsilon,\Delta t}.
\end{align*}
Second, using the identities
\begin{align*}
\frac{\Delta tm_{n+1}^\epsilon}{\epsilon}&=\Delta \beta_n+\epsilon(m_n^{\epsilon,\Delta t}-m_{n+1}^{\epsilon,\Delta t})\\
\zeta^\epsilon(t_{n+1})-\zeta^\epsilon(t_n)&=\int_{t_n}^{t_{n+1}}\frac{m^\epsilon(t)}{\epsilon}dt=\beta(t_{n+1})-\beta(t_n)+\epsilon(m^\epsilon(t_n)-m^\epsilon(t_{n+1}))
\end{align*}
and $\Delta \beta_n=\beta(t_{n+1})-\beta(t_n)$, one obtains
\begin{align*}
Y_{n+1}^{\epsilon,\Delta t}-R_{n,1}^{\epsilon,\Delta t}&=\Phi(\epsilon (m_n^{\epsilon,\Delta t}-m^\epsilon(t_{n+1}))+\Delta \beta_n,X_n^\epsilon)\\
&=\Phi(\Delta \beta_n+\epsilon(m^\epsilon(t_n)-m^\epsilon(t_{n+1}))+\epsilon(m_n^{\epsilon,\Delta t}-m^\epsilon(t_n)),X_n^\epsilon)\\
&=\Phi(\zeta^\epsilon(t_{n+1})-\zeta^\epsilon(t_n)+\epsilon(m_n^{\epsilon,\Delta t}-m^\epsilon(t_n)),X_n^\epsilon)\\
&=\Phi(\zeta^\epsilon(t_{n+1})-\zeta^\epsilon(t_n),\Phi(\epsilon(m_n^{\epsilon,\Delta t}-m^\epsilon(t_n)),X_n^\epsilon))+R_{n,2}^{\epsilon,\Delta t}\\
&=\Phi(\zeta^\epsilon(t_{n+1})-\zeta^\epsilon(t_n),Y_{n}^{\epsilon,\Delta t})+R_{n,2}^{\epsilon,\Delta t},
\end{align*}
using the definition~\eqref{eq:auxY} of $Y_n^{\epsilon,\Delta t}$ in the last step. This concludes the proof of the claim~\eqref{eq:claim}.

Combining~\eqref{eq:errorbound} and~\eqref{eq:claim}, one obtains the following upper bound for the error:
\begin{align*}
\dd_p(Y_N^{\epsilon,\Delta t},X^\epsilon(t_N))&\le C_p(T)\sum_{n=0}^{N-1}\dd_{2p}\Bigl(\Phi\bigl(\zeta^\epsilon(t_{n+1})-\zeta^\epsilon(t_n),Y_{n}^{\epsilon,\Delta t}\bigr),\varphi\bigl(\zeta^\epsilon(t_{n+1})-\zeta^\epsilon(t_n),Y_{n}^{\epsilon,\Delta t}\bigr)\Bigr)\\
&+C_p(T)\sum_{n=0}^{N-1}\dd_{2p}(R_{n,1}^{\epsilon,\Delta t},0)+C_p(T)\sum_{n=0}^{N-1}\dd_{2p}(R_{n,2}^{\epsilon,\Delta t},0).
\end{align*}
To conclude the proof, it remains to prove upper bounds for the three terms in the right-hand side of the inequality above.

$\bullet$ Using the inequality~\eqref{eq:ass_integrator-order} from Assumption~\ref{ass:integrator}, H\"older's inequality, the exponential moment bounds~\eqref{eq:expomoments-zeta-beta}, and finally the inequality~\eqref{eq:incrementszeta}, one obtains, for all $n\in\{0,\ldots,N-1\}$,
\begin{align*}
\dd_{2p}\Bigl(\Phi\bigl(\zeta^\epsilon(t_{n+1})-\zeta^\epsilon(t_n),Y_{n}^{\epsilon,\Delta t}\bigr),\varphi\bigl(\zeta^\epsilon(t_{n+1})-\zeta^\epsilon(t_n),Y_{n}^{\epsilon,\Delta t}\bigr)\Bigr)
&\le C\|\bigl(\zeta^\epsilon(t_{n+1})-\zeta^\epsilon(t_n)\bigr)^3e^{C|\zeta^\epsilon(t_{n+1})-\zeta^\epsilon(t_n)|}\|_{2p}\\
&\le C_p(T)\|\zeta^\epsilon(t_{n+1})-\zeta^\epsilon(t_n)\|_{7p}^3\\
&\le C_p(T)\Delta t^{\frac32}.
\end{align*}
Therefore one has the upper bound
\begin{equation}\label{eq:error0}
\sum_{n=0}^{N-1}\dd_{2p}\Bigl(\Phi\bigl(\zeta^\epsilon(t_{n+1})-\zeta^\epsilon(t_n),Y_{n}^{\epsilon,\Delta t}\bigr),\varphi\bigl(\zeta^\epsilon(t_{n+1})-\zeta^\epsilon(t_n),Y_{n}^{\epsilon,\Delta t}\bigr)\Bigr)\le C_p(T)\Delta t^{\frac12}.
\end{equation}

$\bullet$ Using the inequality~\eqref{eq:lem-integrator} from Lemma~\ref{lem:integrator} with $t_1=\epsilon(m_{n+1}^{\epsilon,\Delta t}-m^\epsilon(t_{n+1}))$ and $t_2=\frac{\Delta tm_{n+1}^{\epsilon,\Delta t}}{\epsilon}$, one obtains, for all $n\in\{0,\ldots,N-1\}$,
\begin{align*}
\dd_{2p}(R_{n,1}^{\epsilon,\Delta t},0)&\le C\|\epsilon(m_{n+1}^{\epsilon,\Delta t}-m^\epsilon(t_{n+1}))\bigl(\frac{\Delta tm_{n+1}^{\epsilon,\Delta t}}{\epsilon}\bigr)^2e^{C|\epsilon(m_{n+1}^{\epsilon,\Delta t}-m^\epsilon(t_{n+1}))|}e^{C|\frac{\Delta tm_{n+1}^{\epsilon,\Delta t}}{\epsilon}|}\|_{2p}\\
&+C\|\epsilon^2(m_{n+1}^{\epsilon,\Delta t}-m^\epsilon(t_{n+1}))^2\frac{\Delta tm_{n+1}^{\epsilon,\Delta t}}{\epsilon}e^{C|\epsilon(m_{n+1}^{\epsilon,\Delta t}-m^\epsilon(t_{n+1}))|}e^{C|\frac{\Delta tm_{n+1}^{\epsilon,\Delta t}}{\epsilon}|}\|_{2p}\\
&\le C_p(T)\|\epsilon(m_{n+1}^{\epsilon,\Delta t}-m^\epsilon(t_{n+1}))\bigl(\frac{\Delta tm_{n+1}^{\epsilon,\Delta t}}{\epsilon}\bigr)^2\|_{3p}\\
&+C_p(T)\|\epsilon^2(m_{n+1}^{\epsilon,\Delta t}-m^\epsilon(t_{n+1}))^2\frac{\Delta tm_{n+1}^{\epsilon,\Delta t}}{\epsilon}\|_{3p},
\end{align*}
using H\"older's inequality and the exponential moment bounds~\eqref{eq:expomoments-m},~\eqref{eq:expomoments-m-discrete} and~\eqref{eq:expomoments-discrete}. Finally, using H\"older's inequality, the error estimate~\eqref{eq:cv_m} from Lemma~\ref{lem:cv_m} and the moment bound~\eqref{eq:incrementsdiscrete}, one obtains, for all $n\in\{0,\ldots,N-1\}$,
\[
\dd_{2p}(R_{n,1}^{\epsilon,\Delta t},0)\le C_p(T)\bigl(\Delta t^{\frac12}+\frac{1}{n+1}\bigr)\Delta t+C_p(T)\bigl(\Delta t+\frac{1}{(n+1)^2}\bigr)\Delta t^{\frac12}.
\]
Therefore one has the upper bound
\begin{equation}\label{eq:error1}
\sum_{n=0}^{N-1}\dd_{2p}(R_{n,1}^{\epsilon,\Delta t},0)\le C_p(T)\Delta t^{\frac12}.
\end{equation}

$\bullet$ Note that $R_{0,2}^{\epsilon,\Delta t}=0$. Using the inequality~\eqref{eq:lem-integrator} from Lemma~\ref{lem:integrator} with $t_1=\zeta^\epsilon(t_{n+1})-\zeta^\epsilon(t_n)$ and $t_2=\epsilon(m_{n}^{\epsilon,\Delta t}-m^\epsilon(t_{n}))$, one obtains, for all $n\in\{1,\ldots,N-1\}$,
\begin{align*}
\dd_{2p}(R_{n,2}^{\epsilon,\Delta t},0)&\le C\|(\zeta^\epsilon(t_{n+1})-\zeta^\epsilon(t_n))\bigl(\epsilon(m_{n}^{\epsilon,\Delta t}-m^\epsilon(t_{n}))\bigr)^2e^{C|\zeta^\epsilon(t_{n+1})-\zeta^\epsilon(t_n)|}e^{C|\epsilon(m_{n}^{\epsilon,\Delta t}-m^\epsilon(t_{n}))|}\|_{2p}\\
&+C\|\bigl(\zeta^\epsilon(t_{n+1})-\zeta^\epsilon(t_n)\bigr)^2\epsilon(m_{n}^{\epsilon,\Delta t}-m^\epsilon(t_{n}))e^{C|\zeta^\epsilon(t_{n+1})-\zeta^\epsilon(t_n)|}e^{C|\epsilon(m_{n}^{\epsilon,\Delta t}-m^\epsilon(t_{n}))|}\|_{2p}\\
&\le C_p(T)\|(\zeta^\epsilon(t_{n+1})-\zeta^\epsilon(t_n))\bigl(\epsilon(m_{n}^{\epsilon,\Delta t}-m^\epsilon(t_{n}))\bigr)^2\|_{3p}\\
&+C_p(T)\|\bigl(\zeta^\epsilon(t_{n+1})-\zeta^\epsilon(t_n)\bigr)^2\epsilon(m_{n}^{\epsilon,\Delta t}-m^\epsilon(t_{n}))\|_{3p}
\end{align*}
using H\"older's inequality and the exponential moment bounds~\eqref{eq:expomoments-m},~\eqref{eq:expomoments-m-discrete} and~\eqref{eq:expomoments-zeta-beta}. Finally, using H\"older's inequality, the error estimate~\eqref{eq:cv_m} from Lemma~\ref{lem:cv_m} and the moment bound~\eqref{eq:incrementszeta}, one obtains, for all $n\in\{1,\ldots,N-1\}$,
\[
\dd_{2p}(R_{n,2}^{\epsilon,\Delta t},0)\le C_p(T)\Delta t^{\frac12}\bigl(\Delta t+\frac{1}{n^2}\bigr)+C_p(T)\Delta t\bigl(\Delta t^{\frac12}+\frac{1}{n}\bigr).
\]
Therefore one has the upper bound
\begin{equation}\label{eq:error2}
\sum_{n=0}^{N-1}\dd_{2p}(R_{n,2}^{\epsilon,\Delta t},0)\le C_p(T)\Delta t^{\frac12}.
\end{equation}

Combining the upper bound~\eqref{eq:errorbound} and the three inequalities~\eqref{eq:error0},~\eqref{eq:error1} and~\eqref{eq:error2}, one obtains
\[
\dd_p(Y_N^{\epsilon,\Delta t},X^\epsilon(t_N))\le C_p(T)\Delta t^{\frac12}.
\]
This concludes the treatment of the second error term in the right-hand side of~\eqref{eq:errordecomp}. Combining the two upper bounds yields~\eqref{eq:main}, which concludes the proof of Theorem~\ref{theo:main}.
\end{proof}

\section{Numerical experiments}\label{sec:num}

The objective of this section is to illustrate Theorem~\ref{theo:main} with numerical experiments. Set $d=1$, $T=1$, $x_0=0$ and $\sigma(x)=\cos(x)$ for all $x\in\T=\T^1=\mathbb{R}/\mathbb{Z}$. The reference time-step size denoted by $h_{\rm ref}=\Delta t_{\rm ref}=2^{-18}$, and the time-step size $h=\Delta t$ takes values in $\{2^{-6},\ldots,2^{-16}\}$. The mean-square error is estimated by averaging the error over $M_s$ samples. The error is represented in logarithmic scales. The integrator $\Phi$ is given by Heun's method: $\Phi(t,x)=x+\frac{t}{2}\bigl(\sigma(x)+\sigma(x+t\sigma(x))\bigr)$ for all $t\in\R$ and $x\in\T$. One observes similar results when using for instance the explicit midpoint method, the numerical results are not reported.

Let us first confirm that the order of convergence of the limiting scheme~\eqref{eq:limitingscheme} is equal to $1$, see Proposition~\ref{propo:orderlimitingscheme}. This is illustrated by Figure~\ref{fig:lim}.
\begin{figure}[h]
\centering
\includegraphics[height=5cm,keepaspectratio]{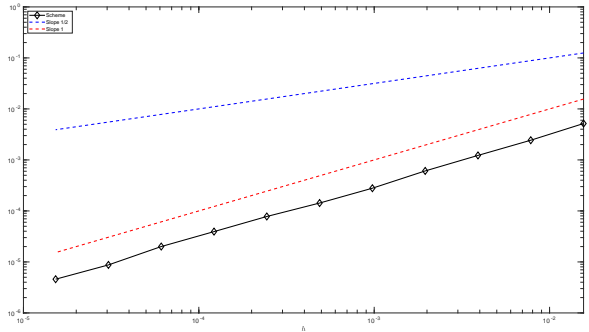}
\caption{Mean-square error as a function of the time-step size $h=\Delta t$ for the limiting scheme. The dotted lines have slopes $1/2$ and $1$.}
\label{fig:lim}
\end{figure}

In Figure~\ref{fig:1}, one has $\epsilon\in\{0.04,0.02,0.01\}$ and $M_s=10^3$. 
\begin{figure}[h]
\centering
\includegraphics[height=7cm,keepaspectratio]{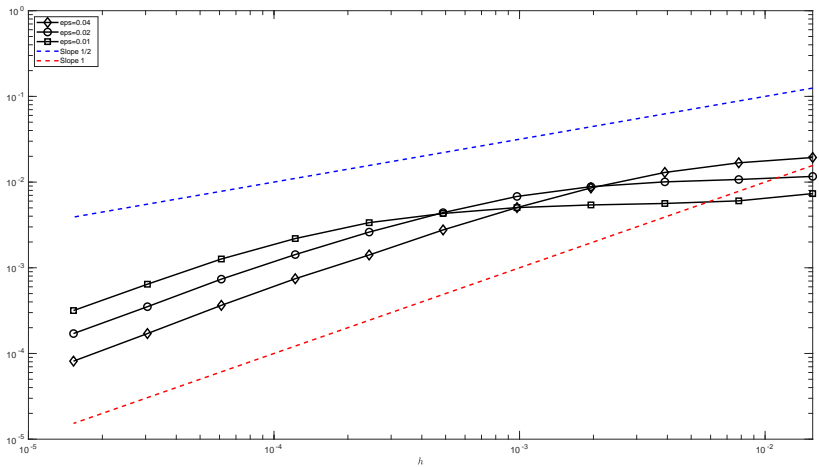}
\caption{Mean-square error as a function of the time-step size $h=\Delta t$, with $\epsilon=0.04,0.02,0.01$. The dotted lines have slopes $1/2$ and $1$.}
\label{fig:1}
\end{figure}
One observes that for large values of $h$, the error decreases when $\epsilon$ decreases, whereas for small values of $h$ the error increases when $\epsilon$ decreases. For any value of $\epsilon$, for sufficiently small values of $h$ the order of convergence is equal to $1$. However, for larger values of $h$ one seems to observe a lower order of convergence.

Let us provide an additional numerical experiment, for different values of $\epsilon$, in order to confirm the results of Figure~\ref{fig:1} and their interpretation. In Figure~\ref{fig:2}, one has $\epsilon\in\{0.1,0.01,0.001\}$ and $M_s=10^2$.
\begin{figure}[h]
\centering
\includegraphics[height=7cm,keepaspectratio]{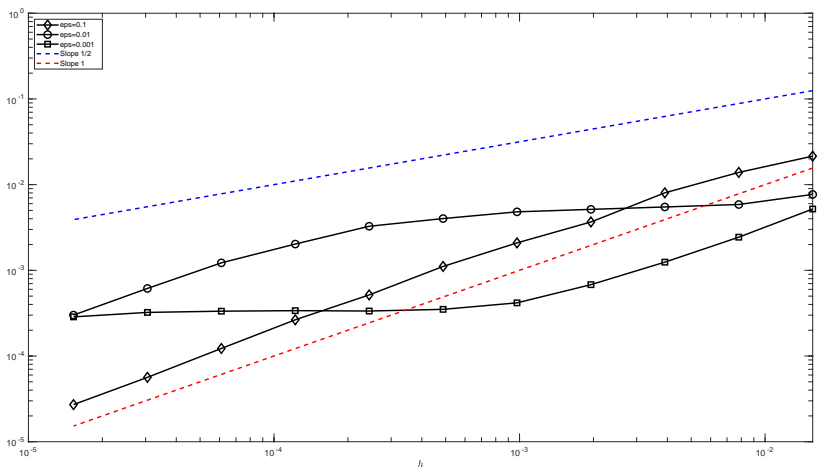}
\caption{Mean-square error as a function of the time-step size $h=\Delta t$, with $\epsilon=0.1,0.01,0.001$.
The dotted lines have slopes $1/2$ and $1$.}
\label{fig:2}
\end{figure}
For the largest value $\epsilon=0.1$, one observes a decrease of the mean-square error with order of convergence $1$. For the other values of $\epsilon$, the behavior is different. For $\epsilon=0.01$, the error saturates for large values of $h$, and decreases with order $1$ when $h$ is sufficiently small -- this is the same behavior as observed in Figure~\ref{fig:1}. For $\epsilon=0.001$, one observes first a decrease with order $1$ for large $h$, and then the error saturates for smaller values of $h$. If one could decrease the values of $h$, one would again observe a decrease of the error with order $1$ for this value of $\epsilon$.

Owing to the results of Figure~\ref{fig:1} and~\ref{fig:2}, it is not possible to replace the order of convergence $1/2$ in the error estimate given in Theorem~\ref{theo:main} by $1$. The behavior of the mean-square error when $\Delta t$ and $\epsilon$ vary is not trivial. However, it is remarkable to be able to obtain a uniform error estimate with respect to $\epsilon$, with order of convergence $1/2$ with respect to $\Delta t$.

Is is also worth providing numerical experiments when the standard explicit Euler scheme is used, i.\,e. $\Phi(t,x)=x+t\sigma(x)$. In that case, as explained in Section~\ref{sec:results-main}, Theorem~\ref{theo:main} does not hold. Figure~\ref{fig:nonAP} gives results with $M_s=10^2$, and with $\epsilon\in\{0.04,0.02,0.01\}$ (left) and $\epsilon\in\{0.1,0.01,0.001\}$ (right).

\begin{figure}[h]
\centering
\includegraphics[height=4cm,keepaspectratio]{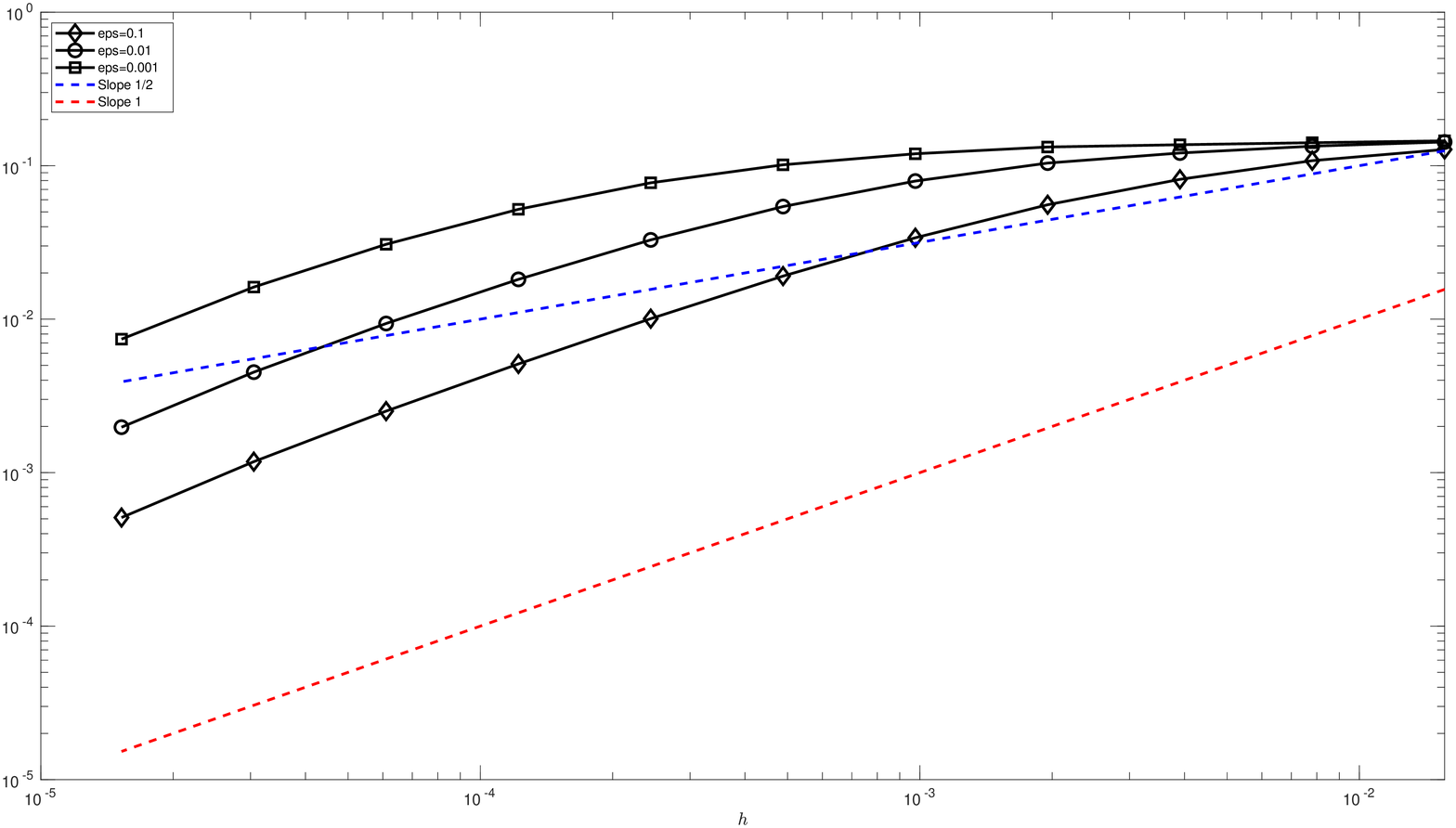}
\includegraphics[height=4cm,keepaspectratio]{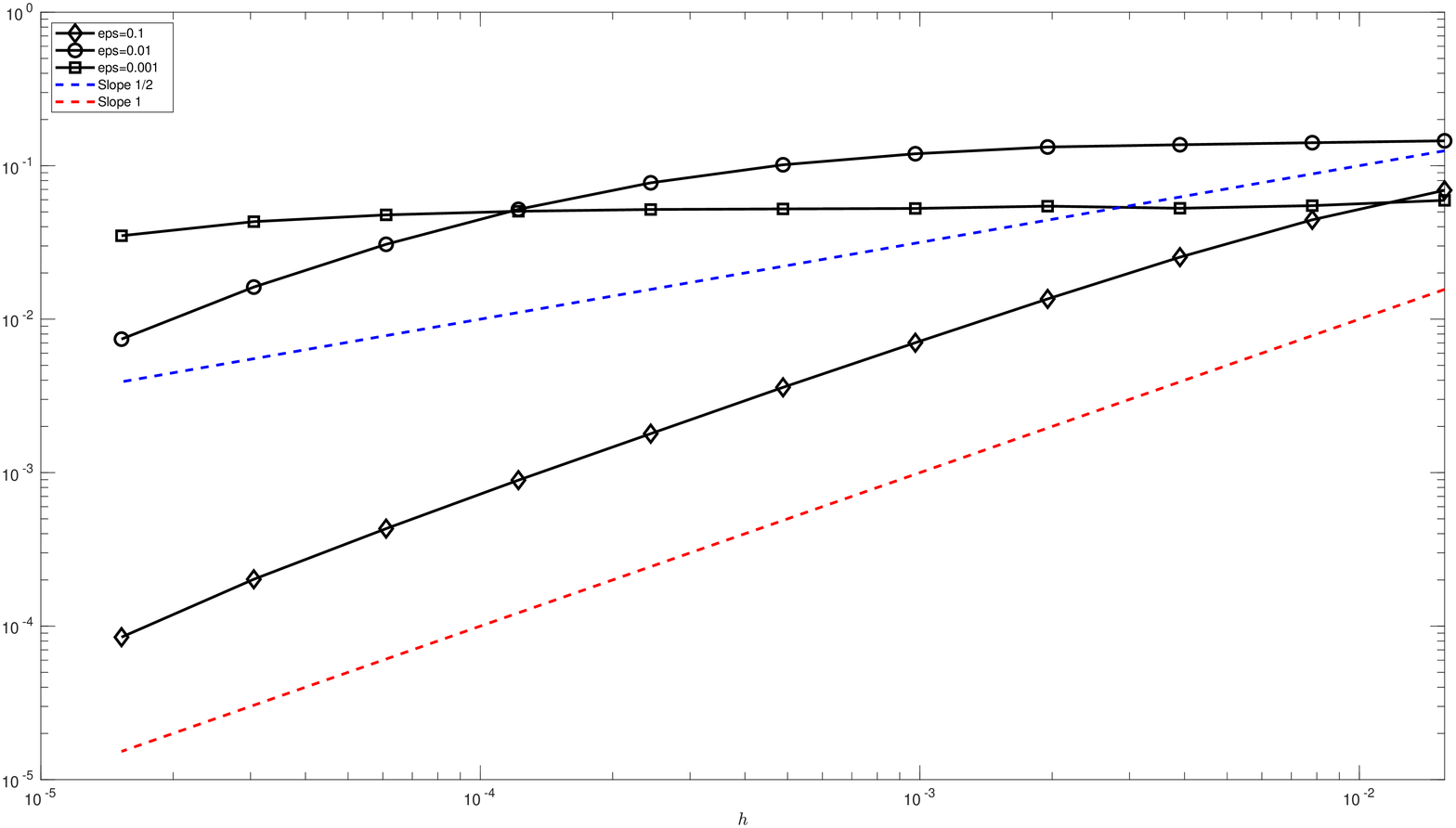}
\caption{Mean-square error for the standard Euler scheme as a function of the time-step size $h=\Delta t$, with $\epsilon=0.04,0.02,0.01$ (left) and $\epsilon=0.1,0.01,0.001$ (right).
The dotted lines have slopes $1/2$ and $1$.}
\label{fig:nonAP}
\end{figure}

In Figure~\ref{fig:nonAP}, one observes that for a given value of $\epsilon$, the error decreases when $h$ is sufficiently small, but the error is large for large values of $h$, contrary to what can be seen in Figures~\ref{fig:1} and~\ref{fig:2}. For $\epsilon=0.001$, the values of $h$ are not sufficiently small to observe the decrease of the error in Figure~\ref{fig:nonAP}. This comparison illustrates the superiory of the scheme~\eqref{eq:scheme} studied in this article for the approximation of the multiscale SDE system~\eqref{eq:SDE}, and how Theorem~\ref{theo:main} is a non trivial theoretical results with a huge importance in practice.

\section*{Acknowledgements}
This work is partially supported by the projects ADA (ANR-19-CE40-0019-02) and SIMALIN (ANR-19-CE40-0016) operated by the French National Research Agency.


\end{document}